\documentclass[12pt,a4paper,final]{amsart}
\usepackage{geometry, comment}
\usepackage{color}
\usepackage{graphicx}
\usepackage{amssymb}
\newtheorem{thm}{Theorem}[section]     
\newtheorem{lem}[thm]{Lemma}
\newtheorem{prop}[thm]{Proposition}
\newtheorem{cor}[thm]{Corollary}

\theoremstyle{definition}  
\newtheorem{definition}[thm]{Definition}

\theoremstyle{remark}
\newtheorem{remark}{{\bf Remark}}[section]
\numberwithin{equation}{section}

\DeclareMathOperator{\tr}{tr}

\DeclareMathOperator{\Lip}{Lip}
\newcommand{\tim}{\times}   
\newcommand{\R}{\mathbb R}
\newcommand{\T}{\mathbb T}

\newcommand{\pl}{\partial}

\newcommand{\cA}{\mathcal{A}}
\newcommand{\cF}{\mathcal{F}}
\newcommand{\cX}{\mathcal{X}}
\newcommand{\N}{\mathbb N}

\newcommand{\lan}{\langle}
\newcommand{\ran}{\rangle}

\newcommand{\disp}{\displaystyle}
\newcommand{\cM}{\mathcal{M}}

\newcommand{\gth}{\theta}
\newcommand{\cR}{\mathcal{R}}

\newcommand{\gd}{\delta}
\newcommand{\gs}{\sigma}

\newcommand{\fr}{\frac}
\newcommand{\rS}{\mathrm{S}}
\newcommand{\rD}{\mathrm{D}}
\newcommand{\rN}{\mathrm{N}}
\newcommand{\gO}{\varOmega}

\def\gl{\lambda}

\def\ga{\alpha}
\def\gl{\lambda}

\def\gz{\zeta}
\def\mid{\,:\,}
\def\du#1{\langle#1\rangle}
\newcommand{\lbar}[1]{\mkern 1.9mu\overline{\mkern-1.9mu#1\mkern-0.1mu}
\mkern 0.1mu}
\def\ol{\lbar}
\usepackage{amsrefs}
\newcommand{\bye}{\end{document}}
\newcommand{\by}{\end{proof}\end{document}}

\def\cL{\mathcal{L}}
\def\bS{\mathbb{S}}
\def\cD{\mathcal{D}}\def\cP{\mathcal{P}}
\def\ep{\varepsilon}
\def\cU{\mathcal{U}}

\def\rN{\mathrm{N}}
\def\cG{\mathcal{G}}
\def\cF{\mathcal{F}}
\def\cR{\mathcal{R}}

\def\t{\T^n}

\def\BC{\mathrm{BC}}
\def\bry{\pl\gO}

\def\tt{\ol\gO}

\def\ran{\rangle}
\def\lan{\langle}

\def\bb{\ol\gO\tim\cA}

\begin{document}

\title[The vanishing discount problem and viscosity Mather measures. 
Part 2]
{The vanishing discount problem and viscosity Mather measures. 
Part 2: boundary value problems}

\author[H. ISHII, H. MITAKE, H. V. TRAN] 
{Hitoshi Ishii${}^*$, Hiroyoshi Mitake and Hung V. Tran}

\thanks{
The work of HI was partially supported by the JSPS grants: KAKENHI \#16H03948,  
\#26220702,  
the work of HM was partially supported by the JSPS grants: KAKENHI \#15K17574, \#26287024, \#16H03948,  
and the work of HT was partially supported by NSF grant DMS-1615944. 
}

\address[H. Ishii]{
Faculty of education and Integrated Arts and Sciences,
Waseda University
1-6-1 Nishi-Waseda, Shinjuku, Tokyo 169-8050 Japan}
\email{hitoshi.ishii@waseda.jp}
\thanks{${}^*$ Corresponding author}

\address[H. Mitake]{
Institute of Engineering, Division of Electrical, Systems and Mathematical Engineering, 
Hiroshima University 1-4-1 Kagamiyama, Higashi-Hiroshima-shi 739-8527, Japan}
\email{hiroyoshi-mitake@hiroshima-u.ac.jp}

\address[H. V. Tran]
{
Department of Mathematics, 
480 Lincoln Dr, Madison, WI 53706, USA}
\email{hung@math.wisc.edu}

\date{\today}

\begin{abstract}
In \cite{IsMtTr1} (Part 1 of this series), we  
have introduced a variational approach to studying
the vanishing discount problem for fully nonlinear, degenerate elliptic, partial differential equations in 
a torus.
We develop this approach further here to handle boundary value problems.
In particular, we establish new representation formulas for solutions of discount problems, critical values,
and use them to prove 
convergence results for the vanishing discount problems.
\end{abstract} 

\maketitle

\tableofcontents

\section{Introduction}
In this paper, we study the fully nonlinear, possibly  
degenerate, elliptic partial differential equation (PDE) in a bounded domain

\begin{equation}\label{DP}\tag{DP$_\lambda$}
\begin{cases}
\gl u(x)+F(x,Du(x),D^2u(x))=0 \ \ \text{ in }\ \gO,&\\[3pt]
\BC \ \ \text{ on }\bry, 
\end{cases}
\end{equation}
where $\gl$ is a given positive constant which is often called a discount factor,
and $F\mid \ol\gO\tim\R^n\tim\bS^n\to\R$ 
is a given continuous function.
Here $\gO$ is a bounded domain (that is, open and connected set) in $\R^n$,
$\bS^n$ 
denotes the space of $n\tim n$ real symmetric matrices,
 and $\BC$ represents a boundary condition (state constraint, Dirichlet, or Neumann boundary condition).
The unknown here is 
a real-valued function $u$ on $\ol \gO$,    
and $Du$ and $D^2u$ denote the gradient and Hessian of $u$, respectively. 
We are always concerned  with viscosity solutions of fully nonlinear, possibly degenerate,  
elliptic PDE, and the adjective ``viscosity'' is omitted henceforth. 

Associated with the vanishing discount problem for (DP$_\lambda$) is the following ergodic problem
\begin{equation}\label{E}\tag{E}
\begin{cases}
F(x,Du(x),D^2u(x))=c \ \ \text{ in }\ \gO,&\\[3pt]
\BC \ \ \text{ on }\bry. 
\end{cases} 
\end{equation}
We refer the boundary-value problem (E), with a given constant $c$,  
as (E$_c$), while the unknown for the ergodic problem (E) 
is a pair of a function $u \in C(\ol \gO)$ and a constant $c \in \R$ 
such that $u$ is a solution of (E$_c$). 
When $(u,c) \in C(\ol \gO) \times \R$
is 
a solution of (E), we call $c$ a critical value (or an additive eigenvalue).

Our main goal is to study the vanishing discount problem for (DP$_\lambda$),
that is, for solutions $v^\lambda$ of (DP$_\lambda$) with $\lambda>0$,
we investigate the asymptotic behavior of $\{v^\lambda\}$ as $\lambda \to 0$.
The particular question we want to address here is whether the {\it whole family} $\{v^\gl\}_{\gl>0}$ (after normalization)
converges or not to a function in $C(\tt)$ as $\gl \to 0$.
As the limiting equation (\ref{E}$_c$) is not strictly monotone in $u$ and has many solutions in general, 
proving or disproving such convergence result is challenging.

The convergence results of the whole family $\{v^\gl\}_{\gl>0}$ were established for 
convex Hamilton-Jacobi equations in \cite{DFIZ} (first-order case in a periodic setting), 
\cite{AlAlIsYo} (first-order case with Neumann-type boundary condition), 
\cite{MiTr} (degenerate viscous case in a periodic setting).
Recently, the authors \cite{IsMtTr1}  
have developed
a new variational approach 
for this
vanishing discount problem for fully nonlinear, degenerate elliptic PDEs,
and 
proved the convergence of the whole family $\{v^\gl\}_{\gl>0}$ in the periodic setting.

We develop the variational method introduced in \cite{IsMtTr1} further here to handle boundary value problems.
Our goal is twofold. 
Firstly, we establish new representation formulas for $v^\gl$ as well as the critical value $c$
in the settings of the state constraint, Dirichlet, and Neumann boundary conditions.
Secondly, we apply these representation formulas to show that $\{v^\gl\}_{\gl>0}$ (after normalization)
converges as $\gl \to 0$.
Let us make it clear that all of the results in this paper and the aforementioned ones 
require 
the 
convexity of $F$ in the gradient and Hessian variables.
We refer to a forthcoming paper \cite{GMT} for convergence results of vanishing discount problems for 
some nonconvex first-order Hamilton-Jacobi equations. 

The main results, which, as mentioned above, consist of representation formulas and the convergence of $\{v^\gl\}_{\gl>0}$, 
are stated in Sections \ref{sec-s}, 
\ref{sec-d}, and \ref{sec-n}  
for the state constraint, Dirichlet, and Neumann
problems, respectively.   
 
\subsection{Setting and Assumptions}
We describe the setting and state the main assumptions here.  

Let $\cA$ be a non-empty, $\gs$-compact and locally compact metric space and 
$F\mid \ol\gO\tim\R^n\tim\bS^n\to\R$ be given by
\begin{equation}\tag{F1}\label{F1}
F(x,p,X)=\sup_{\ga\in\cA}(-\tr a(x,\ga)X-b(x,\ga)\cdot p-L(x,\ga)), 
\end{equation}
where  
$a\mid \ol\gO\tim\cA \to\bS_+^n$, $b\mid \ol\gO\tim\cA\to\R^n$ and $L\mid \ol\gO\tim\cA\to\R$ are continuous.
Here, $\bS_+^n$ denotes the set of all non-negative definite matrices 
$A\in\bS^n$, $\tr A$ and $p\cdot q$ designate the trace of $n\tim n$ matrix $A$ and
the Euclidean inner product of $p,q\in\R^n$, respectively. 

Assume further that
\begin{equation}\tag{F2}\label{F2}
F\in C(\ol\gO\tim\R^n\tim\bS^n).
\end{equation}
It is clear under \eqref{F1} and \eqref{F2} 
that $F$ is degenerate elliptic in the sense 
that for all $(x,p,X)\in \ \ol \gO \tim\R^n\tim\bS^n$, if $Y\in\bS_+^n$,  
then $F(x,p,X+Y)\leq F(x,p,X)$ and that, for each $x\in\ol\gO$, 
the function $(p,X)\mapsto F(x,p,X)$ is convex on $\R^n\tim\bS^n$. 
The equations \eqref{DP}, (\ref{E}$_c$), with $F$ of the form \eqref{F1}, 
are called Bellman equations, or Hamilton-Jacobi-Bellman equations 
in connection with the theory of stochastic optimal control. 
In this viewpoint, the set $\cA$ is often called a control set or region.    

With the function $L$ in the definition of $F$, we define
\[
\Phi^+:=\left\{\phi\in C(\tt\tim\cA)\mid 
\phi(x,\ga)=tL(x,\ga)+\chi(x) \ \text{for some} \ t>0,  \, \chi\in C(\tt)\right\}. 
\]
It is clear that $\Phi^+$ is a convex cone in $C(\tt \tim \cA)$.
For $\phi= tL+\chi \in \Phi^+$, we define
\[
F_\phi(x,p,X) = \sup_{\ga \in \cA} \left(- \tr a(x,\ga)X - b(x,\ga)\cdot p - \phi(x,\ga) \right).
\] 
The form of $\phi$ allows us to compute that
\begin{align*}
F_\phi(x,p,X) &= \sup_{\ga \in \cA} \left(- \tr a(x,\ga)X - b(x,\ga)\cdot p - \phi(x,\ga) \right)\\
&=t \sup_{\ga \in \cA} \left(- \tr a(x,\ga) t^{-1}X - b(x,\ga)\cdot t^{-1}p - L(x,\ga) \right) - \chi(x)\\
&=tF(x,t^{-1}p,t^{-1}X) - \chi(x),
\end{align*}
which yields that $F_\phi \in C(\tt \times \R^n \times \bS^n)$ if we assume (F2). 
We note here that,  
except when $L(x,\ga)$ is independent of $\ga$, 
$\phi\in\Phi^+$ is represented 
uniquely as $\phi=tL+\chi$ for some $t>0$ and $\chi\in C(\tt)$.  
 
We often write $F[u]$ and $F_\phi[u]$ to denote the functions 
$x\mapsto F(x,Du(x),D^2u(x))$ and 
$x\mapsto F_\phi(x,Du(x),D^2u(x))$, respectively. 

The 
following 
are some further assumptions we need 
in the paper,
and we put labels on these for convenience later.

\begin{equation}\tag{CP$_{\rm loc}$}\label{CP}
\left\{\text{
\begin{minipage}{0.83\textwidth}
For any $\gl>0$, $\phi \in \Phi^+$, and open subset $U$ of $\gO$,
if $v,\,w\in C(U)$ are a subsolution and a supersolution of 
$\gl u+F_\phi[u]=0$ in $U$, respectively, and $v \leq w$ on $\partial U$,
then $v\leq w$ in $U$.
\end{minipage}
}\right.\end{equation}
 
\begin{equation} \tag{EC}\label{EC}
\left\{\text{
\begin{minipage}{0.85\textwidth}
For $\gl>0$, let $v^\gl \in C(\tt)$ be a solution of \eqref{DP}. 
The family $\{v^\gl\}_{\gl>0}$ is equi-continuous on $\tt$.
\end{minipage}
}\right.
\end{equation}

\begin{equation}\tag{L}\label{L}\left\{
\text{
\begin{minipage}{0.85\textwidth}
$L=+\infty\,$ at infinity,\ 
that is, for any $M\in\R$, there exists a compact subset $K$ of $\cA$ 
such that $L\geq M$ in $\tt\tim(\cA\setminus K)$.
\end{minipage}}
\right.
\end{equation}

We say that $L$ is coercive if condition \eqref{L} is satisfied, 
and, when $\cA$ is compact, condition \eqref{L} always holds.

Some more assumptions are needed 
and stated in the upcoming sections according to the type of boundary conditions.

\subsection*{Outline of the paper}
In Section \ref{sec-pre}, we introduce some notation and preliminaries.
We first consider the state constraint problem in Section \ref{sec-s} as it is the simplest one.
The Dirichlet problem and the Neumann problem are studied in Section \ref{sec-d} and Section \ref{sec-n}, respectively.
Finally, some examples are given in Section \ref{sec-ex}.

\section{Notation and preliminaries} \label{sec-pre}
First we present our notation.
Given a metric space $E$, $\Lip(E)$ denotes the space of Lipschitz continuous functions on $E$. Also, 
let $C_{\rm c}(E)$ denote the space of continuous 
functions on $E$ with compact support. 
Let $\cR_E$, $\cR_E^+$ and $\cP_E$ denote the spaces 
of all Radon measures, all nonnegative Radon measures and Radon 
probability measures on $E$, respectively. 
For any function $\phi$ on $E$ integrable with respect 
$\mu\in\cR_E$, we write
\[
\lan\mu,\phi\ran=\int_E \phi(x)\mu(d x).
\] 
With the function $L$ from the definition of $F$, we set
\[\begin{aligned}
\Phi^+&\,:=\left\{tL+\chi\mid t>0,  \, \chi\in C(\tt)\right\},\qquad
\Psi^+:=\Phi^+\tim C(\bry),\\
\Psi^+(M)&\,:= \{(tL+\chi,\psi)\in\Psi^+\mid 
\|\chi\|_{C(\tt)}< tM, \|\psi\|_{C(\bry)}< tM\} \ \ \text{ for }M>0,
\end{aligned}\]
where $(t,\chi)$ in the braces above 
ranges over $(0,\,\infty)\tim C(\tt)$,
and
\[
\cR_L:=\{\mu\in\cR_{\bb}\mid L \text{ is integrable with respect to }\mu\}. 
\]

Next, we give three basic lemmas  
related to the weak convergence of measures.  
Let $\cX$ be a $\gs$-compact, locally compact metric space 
and $f\mid \cX\to\R$ be a continuous function. 
Henceforth, when $f$ is bounded from below on $\cX$ 
and $\mu$ is a nonnegative Radon measure on $\cX$, 
we write 
\[
\int_\cX f(x)\mu(dx)=+\infty,
\] 
to indicate that $f$ is not integrable with respect to $\mu$.
We denote by $\lan\mu,f\ran$, for simplicity, the integral
\[
\int_\cX f(x)\mu(dx).
\]

Assume that 
$f$ is coercive in the sense that $f=+\infty$ at infinity. 
In particular, $f$ is bounded below on $\cX$. 

\begin{lem}\label{basic-lsc} The functional 
$\mu\mapsto \lan \mu, f\ran$ 
is lower semicontinuous on $\cR_\cX^+$ 
in the topology of weak convergence of measures. 
\end{lem}

\begin{proof} Let $\{\mu_j\}_{j\in\N}\subset\cR_{\cX}^+$ be a sequence 
converging to $\mu\in\cR^+_{\cX}$ weakly in the sense of measures.
Select a compact set $K\subset\cX$ so that $f(x)>0$ 
for all $x\in\cX\setminus K$, and then a sequence 
$\{\chi_k\}_{k\in\N}\subset C_{\rm c}(\cX)$ so that 
$0\leq \chi_k\leq \chi_{k+1}\leq 1$ on $\cX$, 
$\chi_k=1$ on $K$ for all 
$k\in\N$,
and $\chi_k \to \mathbf{1}_{\cX}$ pointwise as $k \to \infty$.
Since $\chi_k f\in C_{\rm c}(\cX)$ and $\chi_k f\leq f$ on $\cX$, we have
\[
\lan\mu,\chi_k f\ran
=\lim_{j\to\infty}\lan\mu_j,\chi_kf\ran
\leq\liminf_{j\to\infty}\lan\mu_j,f\ran \ \ \ \text{ for all }\ k\in\N.
\]
Hence, using the monotone convergence theorem if
\[
\liminf_{j\to\infty}\lan\mu_j,f\ran<+\infty,
\]
we conclude that
\[
\lan\mu,f\ran
\leq\liminf_{j\to\infty}\lan\mu_j,f\ran. \qedhere
\]
\end{proof}

\begin{lem} \label{basic-cpt}
Let $r\in\R$ and $s>0$, and define 
\[
\cR^+_{f,r,s}=\{\mu\in\cR^+_\cX\mid \lan \mu, f\ran\leq r,\ 
\mu(\cX)\leq s\}
\]
Then $\cR^+_{f,r,s}$ is compact with the topology of weak convergence of measures. 
\end{lem} 

\begin{proof} Let $\{\mu_j\}_{j\in\N}$ be a sequence of 
measures in $\cR^+_{f,r,s}$. 
We need to show that there is a subsequence $\{\mu_{j_k}\}_{k\in\N}$ of $\{\mu_j\}$ that converges to some $\mu\in\cR^+_{f,r,s}$ 
weakly in the sense of measures.

Since $f$ is bounded from below, the sequence 
$\{\lan\mu_j,f\ran\}_{j\in\N}$ is bounded from below, 
and, hence, it is bounded.  
We may thus assume, by passing to a subsequence if needed, that 
the sequence $\{\lan\mu_j,f\ran\}_{j\in\N}$ is convergent.
Since $f=+\infty$ at infinity, the boundedness of the sequence 
$\{\lan\mu_j,f\ran\}_{j\in\N}$ and Chebyshev's inequality 
imply that the family 
$\{\mu_j\}_{j\in\N}$ is tight. Prokhorov's theorem 
guarantees that there is a subsequence $\{\mu_{j_k}\}_{k\in\N}$
of $\{\mu_j\}$ that converges to some $\mu\in\cR^+_\cX$ 
weakly in the sense of measures. 
The weak convergence of $\{\mu_{j_k}\}_{k\in\N}$ 
readily yields $\,\mu(\cX)\leq s$. 
Lemma \ref{basic-lsc} 
ensures that 
\[\lan\mu,f\ran\leq \lim_{k\to\infty}\lan\mu_{j_k},f\ran\leq r,
\]
which then shows that $\mu\in\cR^+_{f,r,s}$. The proof is complete.
\end{proof}

In the next lemma, we consider the case where $\cX=\bb$.

\begin{lem}\label{mod} Let $\mu\in\cR^+_{\bb}$, $m\in\R$ 
and $L\in C(\bb)$. Assume that $\cA$ is not compact, $\mu(\bb)>0$, $m>\lan \mu,L\ran$, and 
that $L$ is coercive. 
Then, there exists $\tilde\mu\in\cR^+_{\bb}$ such that 
\begin{align}
&\tilde\mu(\bb)=\mu(\bb),\tag{i}\\
&\lan\tilde\mu,L\ran=m,\tag{ii}
\\
&\lan\tilde\mu,\psi\ran=\lan\mu,\psi\ran \ \ \ \text{ for all }\ 
\psi\in C(\tt).\tag{iii}
\end{align}
\end{lem}

\begin{proof}
Put
$\, 
m_0:=\lan\mu,L\ran$,
and pick $\ga_1\in\cA$ so that 
\[
\min_{x\in\tt}L(x,\ga_1)\tim \mu(\bb)>m.
\]
Define the Radon measure $\nu$ on $\bb$, through the Riesz representation theorem, by requiring 
\[
\lan\nu,\phi\ran=\int_{\bb}\phi(x,\ga_1)\mu(dxd\ga) 
\ \ \ \text{ for }\ \phi\in C_{\rm c}(\bb),
\]
and note that 
\[
\nu(\bb)=\int_{\bb}\mu(dx d\ga)=\mu(\bb).
\]

We set 
\[
m_1:=\lan\nu,L\ran,
\]
observe that 
\[
m_1\geq \int_{\bb}\min_{\tt}L(\cdot,\ga_1)\mu(dxd\ga)
=\min_{\tt}L(\cdot,\ga_1)\mu(\bb)>m>m_0,
\]
and put
\[
t=\fr{m-m_0}{m_1-m_0}.
\]
It is clear that $t\in (0,\,1)$ and, if we set 
\[
\tilde\mu=(1-t)\mu+t\nu,
\]
then $\tilde\mu$ is a nonnegative Radon measure on $\bb$.

We observe 
\[
\tilde\mu(\bb)=(1-t)\mu(\bb)+t\nu(\bb)=\mu(\bb), 
\]
that
\[
\lan\tilde\mu,L\ran
=(1-t)\lan\mu,L\ran+t\lan\nu,L\ran
=(1-t)m_0+tm_1=m,
\]
and that for any $\psi\in C(\tt)$,
\[\begin{aligned}
\lan\tilde\mu,\psi\ran
&\,=(1-t)\lan\mu,\psi\ran
+t\int_{\bb}\psi(x)\nu(dxd\ga)
\\&\,=(1-t)\lan\mu,\psi\ran
+t\int_{\bb}\psi(x)\mu(dxd\ga)
=\lan\mu,\psi\ran.\end{aligned} 
\]
The proof is complete.
\end{proof}

\section{State constraint problem} \label{sec-s}

We consider  the state constraint boundary problem.
To avoid confusion, we rename discount equation (DP$_\lambda$) as (S$_\gl$), 
and ergodic equation (E) as (ES) in this section.
Here the letter S 
represents ``state constraint".
The two problems of interest are
\[\tag{S$_\gl$}\label{S} 
\begin{cases}
\gl u+F[u]\leq 0 \ \  \text{ in }\ \gO, &\\[3pt]
\gl u+F[u]\geq 0 \ \ \text{ on }\ \lbar\gO,
\end{cases}
\]
for $\gl>0$, and
\[\tag{ES}\label{ES} 
\begin{cases}
F[u]\leq c \ \  \text{ in }\ \gO, &\\[3pt]
F[u]\geq c \ \ \text{ on }\ \lbar\gO.
\end{cases}
\]
Given a constant $c$, we refer as (ES$_c$) the state constraint problem (ES).

We assume in addition the following assumptions, which concern 
with the comparison principle and solvability for \eqref{S}.
\[\tag{CPS}\label{CPS}
\left\{\text{
\begin{minipage}{0.85\textwidth}
The comparison principle holds for \eqref{S} for every $\gl>0$.
More precisely, if $v\in C(\ol \gO)$ is a subsolution of $\gl u + F[u] = 0$ in $\gO$,
and $w\in C(\ol \gO)$ is a supersolution of $\gl u+F[u] = 0$ on $\tt$,
then $v\leq w$ on $\tt$.  
\end{minipage}
}\right.
\]

\[\tag{SLS}\label{SLS}
\text{
\begin{minipage}{0.85\textwidth}
For any $\gl>0$, \eqref{S} admits a solution $v^\gl \in C(\tt)$.
\end{minipage}
}.
\]


\begin{prop}
Assume \eqref{F1}, \eqref{F2}, \eqref{CPS}, \eqref{SLS} and \eqref{EC}.
Then there exists a solution $(u,c)\in C(\tt)\times\R$ of \eqref{ES}. 
Moreover, the constant $c$ is uniquely determined by 
\begin{equation}\label{critical-value-s}
c=\inf\left\{d\in\R\mid \text{there exists} \ v\in C(\tt) \ \text{such that} \ 
F[v]\le d \ \text{in} \ \gO\right\}.  
\end{equation}
\end{prop}

We denote the constant given by \eqref{critical-value-s} by $c_{\rS}$
and call it the critical value of \eqref{ES}.

The proof of the proposition above is somehow standard,
and we give only its outline. 

\begin{proof} For each $\gl>0$ let $v^\gl\in C(\tt)$ be a (unique) solution of \eqref{S} and set $u^\gl:=v^\gl-m_\gl$, 
where $m_\gl:=\min_{\tt}v^\gl$. By \eqref{EC}, the family 
$\{u^\gl\}_{\gl>0}$ is relatively compact in $C(\tt)$.

Set $M=\|v^1\|_{C(\tt)}$,
and observe that for any $\gl>0$, 
the functions $v^1+(1+\gl^{-1})M$ and $v^{1}-(1+\gl^{-1})M$ are 
a supersolution and a subsolution of \eqref{S}, respectively. By the comparison principle \eqref{CPS}, we find that for any $\gl>0$, 
\[
v^{1}-(1+\gl^{-1})M\leq v^\gl\leq v^1+(1+\gl^{-1})M \ \ \text{ on }\tt,
\]
which implies that the collection 
$\{\gl m_\gl\}_{\gl>0}\subset\R$ 
is bounded. 

We may now choose a sequence $\{\gl_j\}_{j\in\N}$ converging to 
zero such that $\{u^{\gl_j}\}_{j\in\N}$ converges to 
a function $u\in C(\tt)$ and $\{\gl_j m_{\gl_j}\}_{j\in\N}$ 
converges to a number $-c^*$. Since $u^{\gl_j}$ is a solution 
of $\gl_j(u^{\gl_j}+m_{\gl_j})+F[u^{\gl_j}]\geq 0$ on $\tt$ 
and $\gl_j(u^{\gl_j}+m_{\gl_j})+F[u^{\gl_j}]\leq 0$ in $\gO$,
we conclude in the limit $j\to\infty$ 
that $(u,c^*)$ is a solution of \eqref{ES}.  

Next we prove formula \eqref{critical-value-s}. We write $d^*$ 
the right 
side of \eqref{critical-value-s}. Since $u$ is a solution of 
(\ref{ES}$_{c^*}$), we get $\,d^*\leq c^*$. 
To show that $d^*\geq c^*$, we suppose $d^*<c^*$, and select 
$(v,d)\in C(\tt)\tim\R$ so that $v$ is a subsolution of 
$F[v]=d$ in $\gO$. By adding $v$ a constant if necessary, we may assume that $v>u$ on $\tt$. We observe that if $\ep>0$ is sufficiently small, then $u$ and $v$ are a supersolution
and a subsolution of \eqref{S}, with $\gl$ replaced by $\ep$ 
and $0$ on its right side replaced by $(c+d)/2$, which means 
that 
the functions $u-(c+d)/(2\ep)$ and $v-(c+d)/(2\ep)$ are 
a supersolution
and a subsolution of \eqref{S}, with $\gl$ replaced by $\ep$. 
By the comparison principle, we get 
\[
v-\fr{c+d}{2\ep}\leq u-\fr{c+d}{2\ep} \ \ \ \text{ on }\ \tt,
\] 
but this is a contradiction, which shows that $d^*\geq c^*$.
\end{proof}

By normalization (replacing $F$ and $L$ by $F-c_{\rS}$ 
and $L+c_{\rS}$, respectively), we often assume 
\begin{equation}\tag{Z}\label{Z}
\text{
\begin{minipage}{0.85\textwidth}
the 
critical value 
$c_{\rS}$ 
of \eqref{ES}  
is zero. 
\end{minipage}
}\end{equation}

\subsection{Representation formulas}
We assume \eqref{Z} in this subsection.

For $z\in \ol \gO$ and $\gl \geq 0$,
we define the sets 
$\cF^{\rS}(\gl)\subset C(\tt\tim\cA)\times C(\tt)$ 
and 
$\cG^{\rS}(z,\gl)\subset C(\tt)$, respectively, as 
\begin{align*}
&\cF^{\rS}(\gl):=\left\{(\phi,u)\in \Phi^+\tim C(\tt)\mid u \ \text{ is a subsolution of } \ \gl u + F_\phi[u] =0 \ \text{ in } \ \gO \right\}, \\
&\cG^{\rS}(z,\gl):=\left\{\phi-\gl u(z)\mid (\phi,u)\in
\cF^{\rS}(\gl) \right\}.
\end{align*}
As $\cG^{\rS}(z,0)$ is independent of $z$, we also write 
$\cG^{\rS}(0)$ for $\cG^{\rS}(z,0)$ for simplicity.

\begin{lem}\label{thm3-sec3}Assume \eqref{F1}, \eqref{F2} 
and \eqref{CP}. 
For any $(z,\gl)\in\tt\tim[0,\infty)$, the set $\cG^{S}(z,\gl)$ is a convex cone with vertex at the origin.
\end{lem}

The proof of this lemma parallels 
that of \cite[Lemma 2.8]{IsMtTr1} and hence is skipped.

For $(z,\gl)\in \tt \tim[0,\,\infty)$, we define the set 
$\cG^{\rS}(z,\gl)'\subset\cR_{L}$ (denoted also by $\cG^{\rS}(0)'$ if $\gl=0$) by 
\[
\cG^{\rS}(z,\gl)':=\{\mu\in \cR_{L}\mid \lan\mu,f\ran\geq 0
\ \ \text{ for all }f\in\cG^{\rS}(z,\gl)\}.  
\]
The set $\cG^{\rS}(z,\gl)'$ is indeed the dual cone of 
$\cG^{\rS}(z,\gl)$ in $\cR_L$.

Set $\cP^\rS:=\cP_{\bb}$, 
and, for any compact subset $K$ of $\cA$, let $\cP_K^\rS$ denote 
the subset of all 
probability measures $\mu\in\cP^\rS$ that have support 
in $\t\tim K$, that is, 
\[
\cP^\rS_K:=\{\mu\in\cP^\rS\mid \mu(\tt\tim K)=1\}.
\]

\begin{thm} \label{thm1-sc} Assume \eqref{F1}, \eqref{F2}, 
\eqref{CP}, \eqref{L}, \eqref{CPS}, \eqref{SLS}, \eqref{EC},
and, if $\gl = 0$, \eqref{Z} as well. 
Let $(z,\gl)\in\tt\tim[0,\,\infty)$ and let $v^\gl\in C(\tt)$ be a 
solution of \eqref{S}. Then   
\begin{equation} \label{sc-min}
\gl v^\gl(z)=\inf_{\mu\in \cP^\rS\cap \cG^{\rS}(z,\gl)'}\lan \mu,L\ran.
\end{equation}
\end{thm} 
 
The proof of this theorem is similar to that of \cite[Theorem 3.3]{IsMtTr1}. However, we present it here, 
because it is a key component of the main result.

\begin{proof} We note that, thanks to \eqref{L}, the term 
\[
\max_{x\in\tt}\min_{\ga\in\cA} L(x,\ga)
\] 
defines a real number.  We choose first a constant 
$L_\gl\in\R$ so that 
\begin{equation}\label{thm1-sc-m2}
L_\gl\geq \max_{x\in\tt}\{\gl v^\gl(x),\min_{\ga\in\cA}L(x,\ga)\},
\end{equation}
and then, in view of \eqref{L}, a compact subset $K_\gl$ of $\cA$ so that 
\begin{equation}\label{thm1-sc-m1}
L(x,\ga)\geq L_\gl \ \ \ \text{ for all }\ (x,\ga)\in\tt
\tim(\cA\setminus K_\gl).
\end{equation}  
We fix a compact set $K\subset\cA$ 
so that $K=\cA$ if $\cA$ is compact,  
or, otherwise, $K\supsetneq K_\gl$.

According to the definition of $\cG^{\rS}(z,\gl)'$, since $(L,v^\gl)\in\cF^{\rS}(\gl)$, we have 
\[
0\leq \lan \mu,L-\gl v^\gl(z)\ran=
\lan \mu,L\ran-\gl v^\gl(z)\ \ \text{ for all }\mu\in\cP^\rS\cap \cG^\rS(z,\gl)',
\]
and, therefore,
\begin{equation}\label{thm1-sc-1}
\gl v^\gl(z)\leq \inf_{\mu\in\cP^\rS\cap \cG^\rS(z,\gl)'}\lan\mu,\,L\ran.
\end{equation} 

Next, we show  
\begin{equation}\label{thm1-sc-1+}
\gl v^\gl(z)\geq \inf_{\mu\in\cP_K\cap \cG^{\rS}(z,\gl)'}\lan\mu,\,L\ran,
\end{equation}
which is enough to prove \eqref{sc-min}. 

We put $K_1:=\cA$ if $K_\gl=\cA$, 
or else, pick $\ga_1 \in K \setminus K_\gl$ and put $K_1 := K_\gl \cup \{\ga_1\}$.
Clearly, $K_1$ is compact and $K_1 \subset K$. 
To prove \eqref{thm1-sc-1+}, we need only to show
\begin{equation}\label{thm1-sc-1++}
\gl v^\gl(z)\geq \inf_{\mu\in\cP_{K_1}\cap \cG^{\rS}(z,\gl)'}\lan\mu,\,L\ran.
\end{equation}

Suppose by contradiction that \eqref{thm1-sc-1++} 
is false, which means that
\[
\gl v^\gl(z)<\inf_{\mu\in\cP_{K_1}\cap \cG^{\rS}(z,\gl)'}\lan \mu, L\ran.
\]
Pick $\ep>0$ sufficiently small such that
\begin{equation} \label{thm1-sc-2}
\gl v^\gl(z) + \ep<\inf_{\mu\in\cP_{K_1}\cap 
\cG^{\rS}(z,\gl)'}\lan \mu, L\ran.
\end{equation}

Since $\cG^{\rS}(z,\gl)$ is a convex cone with vertex at the origin, 
we infer that 
\[
\inf_{f\in\cG^{\rS}(z,\gl)}
\lan \mu,f\ran=
\begin{cases} 0 \ \ &\text{ if }\ \mu\in\cP_{K_1} \cap\cG^{\rS}(z,\gl)', \\
-\infty &\text{ if }\ \mu\in \cP_{K_1} \setminus
\cG^{\rS}(z,\gl)'.
\end{cases}
\]
and, hence, the right hand side of \eqref{thm1-sc-2} can be rewritten as 
\begin{equation}\begin{aligned}
\label{thm1-sc-2+} 
\inf_{\mu\in\cP_{K_1} \cap \cG^{\rS}(z,\gl)'}\lan \mu, L\ran
&\,=
\inf_{\mu\in\cP_{K_1}}\
\Big(\lan \mu, L\ran
- \inf_{f\in\cG^{\rS}(z,\gl)}\lan\mu,f\ran
\Big)
=\inf_{\mu\in\cP_{K_1}}\
\sup_{f\in\cG^{\rS}(z,\gl)}\,\lan \mu, L-f\ran.
\end{aligned}
\end{equation} 
Since $\cP_{K_1}$ is a convex compact space, with the topology of weak convergence of 
measures, 
we apply Sion's minimax theorem, to 
get
\[
\min_{\mu\in\cP_{K_1}}
\ \sup_{f\in\cG^{\rS}(z,\gl)}\,\lan \mu, L-f\ran
= \sup_{f\in\cG^{\rS}(z,\gl)}\ \min_{\mu\in\cP_{K_1}}
\,\lan \mu, L-f\ran. 
\]
In view of this and \eqref{thm1-sc-2}, we can pick $(\phi,u) \in \cF^{\rS}(\gl)$ such that 
\begin{equation}\label{thm1-sc-3}
\gl v^\gl(z)+\ep < \lan \mu, L-\phi + \gl u(z)  \ran \ \ \text{ for all } \mu \in \cP_{K_1},
\end{equation}
and $\phi=tL+\chi$ for some $t>0$ and $\chi \in C(\ol \gO)$.

We now prove that there exists $\gth>0$ such that $w:=\gth u$ is a 
subsolution of 
\begin{equation}\label{thm1-sc-4}
\gl w+F[w]=-\gl (v^\gl-w)(z)-\gth\ep\ \ \text{ in }\gO.
\end{equation}  
Once this is done, 
we immediately arrive at a contradiction.
Indeed, if $\gl>0$, then 
$\gz:=w+{(v^\gl-w)(z)}+\gl^{-1}\gth\ep$ is a subsolution of $\gl \gz+F[\gz]=0$ in $\gO$,
and comparison principle \eqref{CPS} 
yields $\gz\leq v^\gl$, which, after evaluation at $z$, gives 
$\,\gl^{-1}\gth\ep\leq 0$, a contradiction. On the other hand, if $\gl=0$, then we 
set $\gz:=w+C$, with a constant $C>\min_{\tt}(v^\gl-w)$, 
choose a constant $\gd>0$ 
sufficiently small so that 
\[
\gd\min_{\tt}v^\gl\geq -\gth\ep/2 \ \ \text{ and } \ \ \gd\max_{\tt} \gz\leq \gth\ep/2,
\] 
observe that the functions 
$\gz$ 
and $v^\gl$ are 
a subsolution and a supersolution of $\gd u+F[u]=-\gth\ep/2$, respectively, in $\gO$ 
and on $\tt$, and, by \eqref{CPS}, get $\gz\le v^\gl$ on $\tt$, 
which is a contradiction.

To show \eqref{thm1-sc-4}, 
we consider the two cases separately.  
The first case is when $K_1=\cA$, and then, 
the Dirac measure $\gd_{(x,\ga)}$ belongs to $\cP_{K_1}$ 
for any $(x,\ga) \in \ol \gO \times \cA$,
which together with \eqref{thm1-sc-3} 
yields 
\[
\gl v^\gl(z) + \ep < L - \phi + \gl u(z) \quad \text{on } \ol \gO \times \cA.
\]
Now, we have 
\[
F(x,p,X) \leq F_\phi(x,p,X) + \gl(u-v^\gl)(z) - \ep \quad \text{for} \ (x,p,X) \in \tt \times \R^n \times \bS^n.
\]
We choose $\theta=1$ and observe that $w=u$ is a 
subsolution of \eqref{thm1-sc-4}.

The other case is that when $K_1=K_\gl \cup\{\ga_1\}$, with 
$\ga_1 \in K \setminus K_\gl$.
As $\gd_{(x,\ga)}\in\cP_{K_1}$ for all $(x,\ga)\in\tt\tim K_1$, we observe, in light of \eqref{thm1-sc-3}, that
\begin{equation}\label{thm1-sc-6}
\gl v^\gl(z)+\ep<(1-t)L(x,\ga)-\chi(x)+\gl u(z) \ \ \text{ for all }(x,\ga)\in\tt\tim K_1. 
\end{equation}

We subdivide the argument into two cases. 
Consider first the case when $t \leq 1$. 
Minimize both sides of \eqref{thm1-sc-6} in $\ga \in K_\gl$ 
and note by \eqref{thm1-sc-m2} and \eqref{thm1-sc-m1} that
$\min_{\ga\in K_\gl}L(x,\ga)\leq L_\gl$, to get
\[\begin{aligned}
\gl v^\gl(z)+\ep&\,<(1-t)\min_{\ga \in K_\gl} L(x,\ga)-\chi(x)+\gl 
u(z)  
\\&\,\leq (1-t)L_\gl-\chi(x)+\gl 
u(z) \ \ \text{ for all }  x \in\tt.
\end{aligned}\]
We use this, \eqref{thm1-sc-m1}, and \eqref{thm1-sc-6}, to deduce that
\[
\gl v^\gl(z)+\ep<(1-t)L(x,\ga)-\chi(x)+\gl 
u(z) \ \ \text{ for all }(x,\ga)\in\tt\tim \cA. 
\] 
From this, we observe
\[
\phi =L+(t-1)L+\chi<L-\gl (v^\gl-u)(z)-\ep \ \text{ on }\bb,
\]
and, hence, that $u$ is a subsolution of \eqref{thm1-sc-4}, with $\gth=1$.

Next, we consider the case when $t \geq 1$. 
By \eqref{thm1-sc-m1} and \eqref{thm1-sc-6}, we get
\[L(x,\ga_1) \geq L_\gl \ \ \text{ and } \ \  
(t-1)L(x,\ga_1)+\chi(x)<-\gl (v^\gl-u)(z)-\ep \ \ \text{ for all }x\in\tt, 
\]
which together yield
\[
(t-1) L_\gl + \chi(x) \leq -\gl(v^\gl-u)(z) - \ep \ \ \text{ for all }  x \in \tt.
\]
We take advantage of \eqref{thm1-sc-m2} to get further that
\[
(t-1) \gl v^\gl(z) + \chi(x) \leq -\gl(v^\gl-u)(z) - \ep \ \ \text{ for all } x \in\tt.
\]
Therefore,
\[
\chi < - t \gl v^\gl(z) + \gl u(z) - \ep,
\]
and
\[
\phi=tL+\chi < tL + t\gl (u/t-v^\gl)(z) - \ep \ \ \text{ in } \tt \times \cA.
\]
From this we deduce that $w:=u/t$ is a subsolution of 
\[
\gl w +F[w]=-\gl(v^\gl-w)(z)-\ep/t \ \ \ \text{ in }\gO.
\]
This completes the proof. 
\end{proof}

We remark that, in the proof above, we have proven the identity
\[
\gl v^\gl(z)=\inf_{\mu\in\cP_K\cap\cG^\rS(z,\gl)'}\lan\mu,L\ran
\]
for a compact set $K\subset\cA$, 
which is a stronger claim than \eqref{sc-min}.

The minimization problem \eqref{sc-min}, in Theorem \ref{thm1-sc}, has minimizers as stated in the upcoming corollary, 
Corollary \ref{cor1-sc}. 

\begin{lem} \label{sc-cpt} 
Assume \eqref{F1}, \eqref{F2} and \eqref{L}. 
Fix a point $z\in\tt$ and a sequence $\{\gl_j\}_{j\in\N}\subset[0,\,\infty)$.  Let 
$\{\mu_j\}_{j\in\N}$ 
be a sequence of measures such that $\mu_j\in\cP^\rS\cap
\cG^\rS(z,\gl_j)'$ for all $j\in\N$. Assume that for some $\rho\in\R$ and 
$\gl\in[0,\,\infty)$, 
\[
\lim_{j\to\infty}\lan\mu_j,L\ran=\rho \ \ \text{ and } \ \ 
\lim_{j\to\infty}\gl_j=\gl. 
\]
Then there exist a measure $\nu\in\cP^\rS\cap
\cG^\rS(z,\gl)'$ 
and a subsequence $\{\mu_{j_k}\}_{k\in\N}$ of 
$\{\mu_j\}$ such that $\,\lan\nu,L\ran=\rho\,$ and  
\[
\lim_{k\to\infty}\lan\mu_{j_k},\psi\ran=\lan\nu,\psi\ran \ \ \text{ for all }\ \psi\in C(\tt).
\]
\end{lem}

\begin{proof} Lemma \ref{basic-cpt} implies that 
$\{\mu_j\}_{j\in\N}$ has a  
subsequence $\{\mu_{j_k}\}_{k\in\N}$ convergent in the topology 
of weak convergence of measures. Let $\mu_0\in\cR^+_{\bb}$ denote 
the limit of $\{\mu_{j_k}\}_{k\in\N}$. It follows 
from the weak convergence of $\{\mu_{j_k}\}$ that $\mu_0$ 
is a probability measure on $\bb$. Hence, $\mu_0\in\cP^\rS$. 
By the lower semicontinuity of 
the functional $\mu\mapsto\lan\mu,L\ran$ on $\cR^+_{\bb}$,
as claimed by Lemma \ref{basic-lsc}, we find that $\, \rho\geq 
\lan\mu_0,L\ran$. 
 
If $\rho>\lan\mu_0,L\ran$, then $\cA$ is not compact and
Lemma \ref{mod} ensures that there is $\tilde\mu_0\in\cP^\rS$ 
such that $\lan\tilde\mu_0,L\ran=\rho$ and 
$\lan\tilde\mu_0,\psi\ran=\lan\mu_0,\psi\ran$ for all 
$\psi\in C(\tt)$. We define $\nu\in\cP^\rS$ by setting 
$\nu=\tilde\mu_0$ if $\rho>\lan\mu_0,L\ran$ and 
$\nu=\mu_0$ otherwise, that is, if $\rho=\lan\mu_0,L\ran$. 
The measure $\nu\in\cP^\rS$ verifies that
$\lan\nu,L\ran=\rho$ and 
$\lan\nu,\psi\ran=\lan\mu_0,\psi\ran$ for all 
$\psi\in C(\tt)$. It follows from the last identity
that 
\[
\lim_{k\to\infty}\lan\mu_{j_k},\psi\ran=\lan\nu,\psi\ran 
\ \ \ \text{ for all }\ \psi\in C(\tt).
\]
 
It remains to check that $\nu\in\cG^\rS(z,\gl)'$. 
Let $f\in\cG^\rS(z,\gl)$, and select 
$(\phi,u)\in\cF^\rS(\gl)$, $t>0$ and $\chi\in C(\tt)$ 
so that $f=\phi-\gl u(z)$ and $\phi=tL+\chi$. 

Fix $j\in\N$, note that $(\phi+(\gl_j-\gl)u,u)
\in\cF^\rS(\gl_j)$ and get
\[
0\leq \lan\mu_j,\phi+(\gl_j-\gl)u-\gl_j u(z)\ran
=t\lan\mu_j,L\ran+
\lan\mu_j,\chi+(\gl_j-\gl)u-\gl_j u(z)\ran,
\]   
because $\mu_j\in\cG^\rS(z,\gl_j)'$. Sending $j\to\infty$ yields
\[
0\leq t\rho 
+\lan\mu_0,\chi-\gl u(z)\ran.
\]
Since $\rho=\lan\nu,L\ran$ and $\lan\mu_0,\chi\ran
=\lan\nu,\chi\ran$, we find that
\[
0\leq t\lan\nu,L\ran+\lan\nu,\chi-\gl u(z)\ran
=\lan\nu,\phi-\gl u(z)\ran,
\]
which shows that $\nu\in\cG^\rS(z,\gl)'$, completing the proof.   
\end{proof}

\begin{cor}\label{cor1-sc} Under the hypotheses of Theorem 
\ref{thm1-sc}, we have
\begin{equation}\label{cor1-sc-1}
\gl v^\gl(z)=\min_{\mu\in\cP^\rS\cap\cG^\rS(z,\gl)'}
\lan\mu,L\ran. 
\end{equation}
\end{cor}

\begin{proof} In view of Lemma \ref{sc-cpt}, we need only to 
show that there exists a sequence 
$\{\mu_j\}_{j\in\N}\subset \cP^\rS\cap\cG^\rS(z,\gl)'$
such that
\[
\lim_{j\to\infty}
\lan\mu_j,L\ran
=\inf_{\mu\in\cP^\rS\cap\cG^\rS(z,\gl)'}\lan\mu,L\ran,
\]
but this is obviously true. 
\end{proof}

\begin{remark} \label{rem1-s}
In the generality that \eqref{Z} is not assumed in Corollary \ref{cor1-sc}, we have
\[
-c_{\rS}=\inf_{\mu\in \cP^\rS\cap \cG^{\rS}(0)'}\lan \mu,L\ran.
\]
\end{remark}

\begin{definition}
We denote the set of minimizers of \eqref{cor1-sc-1} 
by $\cM^{\rS}(z,\gl)$ for $\gl>0$,
and by $\cM^{\rS}(0)$ for $\gl=0$. 
We call any $\mu \in \cM^{\rS}(0)$ a viscosity Mather measure,
and, if $\gl>0$, 
any 
$\gl^{-1}\mu$, with $\mu \in \cM^{\rS}(z,\gl)$ and $\gl>0$, a viscosity Green measure.
\end{definition}

\subsection{Convergence with vanishing discount}
The following theorem is our main result on the vanishing
discount problem for the state constraint problem.
 
\begin{thm}\label{thm2-sc} 
Assume \eqref{F1}, \eqref{F2}, \eqref{CP}, \eqref{L}, 
\eqref{CPS}, \eqref{SLS}, and \eqref{EC}. 
For each $\gl>0$, let $v^\gl\in C(\tt)$ be 
the unique solution of \eqref{S}. 
Then, the family $\{v^\gl+\gl^{-1}c_{\rS}\}_{\gl>0}$ converges to a function $u$ in $ C(\tt)$ as $\gl\to 0$. 
Furthermore, $(u,c_{\rS})$ is a solution of \eqref{ES}. 
\end{thm}

In order to prove this theorem, we need 
the next 
lemma. 

\begin{lem}\label{lem2-sc}
Assume \eqref{F1}, \eqref{F2}, \eqref{CPS}, \eqref{SLS}, \eqref{EC} and \eqref{Z}. 
For each $\gl>0$, let $v^\gl \in C(\tt)$ be the unique solution of \eqref{S}.
Then $\{v^\gl\}_{\gl>0}$ is uniformly bounded on $\tt$.
\end{lem}

\begin{proof} 
Let $u \in C(\tt)$ be a solution of 
(\ref{ES}$_0$).
It is clear that $u+\|u\|_{C(\tt)}$ and $u-\|u\|_{C(\tt)}$ are a supersolution and a subsolution of \eqref{S} respectively for any $\gl>0$.
By the comparison principle, we get
\[
u-\|u\|_{C(\tt)} \leq v^\gl \leq u+\|u\|_{C(\tt)} \ \ \text{on } \tt,
\]
which yields $\|v^\gl\|_{C(\tt)} \leq 2 \|u\|_{C(\tt)}$.
\end{proof}

We are now ready to prove the main convergence result in this section.

\begin{proof}[Proof of Theorem \ref{thm2-sc}]
As always, 
we assume $c_{\rS}=0$.
Let $\cU$ be the set of accumulation points in $C(\tt)$ of $\{v^\gl\}_{\gl>0}$ as $\gl \to 0$. 
By Lemma \ref{lem2-sc} and \eqref{EC}, $\{v^\gl\}_{\gl>0}$ is relatively compact in $C(\tt)$. 
Clearly, $\cU \neq \emptyset$ and any $u\in\cU$ is a solution of 
(\ref{ES}$_0$). Our goal is achieved when 
we prove that $\cU$ has a unique element,  
or equivalently, for any $v, w \in \cU$,
\begin{equation} \label{thm2-sc-1}
v \geq w \  \ \text{ on } \tt.
\end{equation}

Fix any $v,w \in \cU$.
There exist two sequences $\{\gl_j\}$ and $\{\gd_j\}$ 
of positive numbers converging to $0$
such that $v^{\gl_j} \to v$, $v^{\gd_j} \to w$ in $C(\tt)$ as $j \to \infty$.

Fix $z \in \tt$.
By Corollary \ref{cor1-sc}, there exists a 
sequence $\{\mu_j\}_{j\in\N}$ of 
measures 
such that $\mu_j \in \cM^{\rS}(L,z,\gl_j)$ for every 
$j \in \N$.
Since
\[
\lan\mu_j,L\ran=\gl_j v^{\gl_j}(z) \to 0=c_\rS\ \ \text{ as }\ 
j\to\infty, 
\]
Lemma \ref{sc-cpt} guarantees that 
there is $\mu\in\cP^\rS\cap\cG^\rS(0)'$ such that 
\[
\lan\mu,L\ran=0 \ \ \ \text{ and } \ \ \ 
\lim_{j\to\infty}\lan\mu_j,\psi\ran=\lan\mu,\psi\ran
\ \  \ \text{ for all }\ \psi\in C(\tt). 
\]

We note that $(L-\gd_j v^{\gd_j}, v^{\gd_j}) \in \cF^{\rS}(0)$ and $(L+\gl_j w, w) \in \cF^{\rS}(\gl_j)$, which implies that
\[
0 \leq \lan \mu, L-\gd_j v^{\gd_j} \ran 
= -\gd_j \lan \mu, v^{\gd_j} \ran,
\]
and
\[
0 \leq \lan \mu_j, L + \gl_j w - \gl_j w(z) \ran = \gl_j (v^{\gl_j} - w)(z) + \gl_j \lan \mu_j, w \ran.
\]
Dividing the above inequalities by $\gd_j$ and $\gl_j$, respectively, and letting $j \to \infty$ yield 
\[
\lan \mu, w \ran \leq 0 \quad \text{and} \quad 0 \leq (v-w)(z) + \lan \mu, w \ran,
\]
and thus, $v(z) \geq w(z)$. This completes the proof.
\end{proof}

\section{Dirichlet problem} \label{sec-d}

We consider the Dirichlet problem in this section.
We rename \eqref{DP} and \eqref{E} as \eqref{D} and \eqref{ED}, respectively, in 
which the letter D 
refers to ``Dirichlet".
For a given $g\in C(\bry)$, the two problems 
of interest are
\begin{equation}\tag{D$_\gl$}\label{D}
\begin{cases}
\gl u+F[u]= 0 \ \  \text{ in }\ \gO, &\\[3pt]
u=g \ \ \text{ on }\ \bry,
\end{cases}
\end{equation}
for $\gl>0$,
and
\[\tag{ED}\label{ED}  
\begin{cases}
F[u]= c \ \  \text{ in }\ \gO, &\\[3pt]
u=g \ \ \text{ on }\ \pl\gO.
\end{cases}
\]
As usual, (ED$_c$) refers the Dirichlet problem (ED), with a given 
constant $c$.

The function $g\in C(\tt)$ is fixed in the following argument, 
while we also consider the Dirichlet problem 
\[\tag{D$_{\gl,\phi,\psi}$}\label{D'}
\begin{cases}
\gl u+F_\phi[u]=0 \ \ \text{ in }\ \gO,\\[3pt]
u=\psi \ \ \ \text{ on }\ \bry,   
\end{cases}
\]
where $\gl\geq 0$, $\phi\in C(\tt\tim\cA)$ and $\psi\in C(\tt)$ are all given.  

Throughout this paper, we understand that $u\in C(\tt)$ is a subsolution 
(resp., a supersolution) of 
\eqref{D'}, with $\gl\geq 0$, if 
it is a subsolution of $\gl u+F_\phi[u]=0$ in $\gO$ 
in the viscosity sense and verifies $u\leq \psi$ pointwise on $\pl\gO$ 
(resp., a supersolution 
of \eqref{D'} in the viscosity sense). 
As always, we call $u\in C(\tt)$ a solution of 
\eqref{D'} if it is a subsolution and supersolution of \eqref{D'}.

We assume in addition the following conditions.
\[\tag{CPD}\label{CPD}
\left\{\text{
\begin{minipage}{0.85\textwidth} 
The comparison principle holds for \eqref{D'}, with 
$\phi=L+\chi$, for any $\gl>0$, $\chi\in C(\tt)$ and $\psi\in C(\bry)$. 
That is, 
for any subsolution $u\in C(\tt)$ and supersolution $v\in C(\tt)$  
of \eqref{D'}, with $\phi=L+\chi$,  
the inequality $\,u\leq v\,$ holds on $\tt$.
\end{minipage}
}\right.
\]

We remark here that, if $\phi=L+\chi$, then the equation 
$\,\gl u+F_\phi[u]=0\,$ can be written as $\gl u+F[u]=\chi$.  
\[\tag{SLD}\label{SLD}
\left. 
\text{
\begin{minipage}{0.85\textwidth}
For every $\gl>0$, \eqref{D} 
admits a solution $v^\gl \in C(\tt)$.
\end{minipage}
}\right.
\]

In the vanishing discount problem for the Dirichlet 
problem, the state constrain problem comes into play as the following 
results indicate.

\begin{prop}\label{prop1-d0}
Assume \eqref{F1}, \eqref{F2}, \eqref{CPD} and \eqref{SLD}. 
For $\gl>0$, let $v^\gl\in C(\tt)$ be the solution of \eqref{D}. 
Then, \emph{(i)}\ 
if \emph{(ED$_0$)} has a solution in $C(\tt)$, then 
\[
\lim_{\gl\to 0+ }\gl v^\gl(x)=0 \ \ \ \text{ uniformly on }\ \tt,
\]
and \emph{(ii)}\ if 
\emph{(ES$_c$)}, with $c>0$, has a solution in $C(\tt)$, then 
\[
\lim_{\gl\to 0+}\gl v^\gl(x)=-c \ \ \ \text{ uniformly on }\ \tt.
\] 
\end{prop} 

Notice that, in the proposition above, the assumption for the claim (i) 
and that for (ii) are mutually exclusive, since the conclusions are exclusive of one another.

\begin{proof}
Assume first that (ED$_0$) has a solution in $C(\tt)$, which 
denote by $u\in C(\tt)$. We select $M>0$ so that 
$\|u\|_{C(\tt)}\leq M$ and note that the function 
$u+M$ is nonnegative on $\tt$ and hence a supersolution of \eqref{D} for any $\gl>0$
and, similarly, that $u-M$ is a subsolution of \eqref{D} for any $\gl>0$. 
Thus, by \eqref{CPD}, we have $u-M\leq v^\gl\leq u+M$ on $\tt$, which readily yields 
\[
\lim_{\gl\to 0}\gl v^\gl(x)=0 \ \ \ \text{ uniformly on }\ \tt.
\]

Next, assume that (ES$_c$), with $c>0$, has a solution in $C(\tt)$, and 
let $u\in C(\tt)$ be such a solution.
Choose $M>0$ large enough so that $\|u\|_{C(\tt)}+\|g\|_{C(\bry)}\leq M$.
Observe that, for any $\gl>0$,  $u+M-\gl^{-1}c$ is a supersolution of 
$\gl w+F[w]=0$ on $\tt$ and, therefore, is a supersolution of \eqref{D}
and that $u-M-\gl^{-1}c$ is a subsolution of \eqref{D} for any $\gl>0$. 
Hence, we get $u-M-\gl^{-1}c\leq v^\gl\leq u+M-\gl^{-1}c$ on $\tt$ for all $\gl>0$.
This shows that 
\[
\lim_{\gl\to 0}\gl v^\gl(x)=-c.
\]
The proof is now complete. 
\end{proof}

\begin{prop}\label{prop1-d}
Assume \eqref{F1}, \eqref{F2}, \eqref{CPD}, \eqref{SLD} and \eqref{EC}. 
Then, there is a dichotomy: either problem \emph{(ED$_0$)}, or 
\emph{(ES$_c$)}, with some $c>0$, has a solution in $C(\tt)$.  
\end{prop}

Here is an illustrative, simple example regarding the solvability of \eqref{ED} or \eqref{ES}. Let $n=1$ and $m\in\R$, 
and consider the case 
$\gO=(-1,\,1)$, $F(x,p)=|p|+m$, and $g=0$. 
It is easily checked that conditions \eqref{F1}, \eqref{F2},
\eqref{CPD}, \eqref{SLD}, and \eqref{EC} are satisfied.  
The Dirichlet problem
\begin{equation}\label{ex-d-1}
|u'|+m=c \ \ \text{ in }(-1,\,1)\quad \text{ and }\quad
u(-1)=u(1)=0, 
\end{equation}
with $c\in\R$, has a solution $u(x):=(c-m)(1-|x|)$ 
if and only if $c\geq m$. Moreover, if $c=m$, 
then any function $u(x):=C$, with $C\leq 0$, 
is a solution of \eqref{ex-d-1} and any constant function $u$ 
is a solution of the state constraint problem
\begin{equation}\label{ex-d-2}
|u'|+m\geq c \ \ \text{ in }[-1,\,1]
\quad\text{ and }\quad
|u'|+m\leq c \ \ \text{ in }(-1,\,1).
\end{equation}
Thus, if $m\leq 0$, then problem \eqref{ex-d-1}, with $c=0$, has a 
solution in $C([-1,\,1])$, and if $m > 0$, then 
problem \eqref{ex-d-2}, with $c=m$, 
has a solution in $C([-1,\,1])$.

\begin{proof} For $\gl>0$, let $v^\gl\in C(\tt)$ be the solution of \eqref{D}. 
Choose a constant $M>0$ so that  
$\|g\|_{C(\bry)}+\|F(\cdot,0,0)\|_{C(\tt)}\leq M$
and observe that, if $\gl>0$, then 
the constant functions $\gl^{-1}M$ and 
$-\gl^{-1}M$ are, respectively, a supersolution and a subsolution of \eqref{D}. 
By \eqref{CPD}, we have $-\gl^{-1}M\leq v^\gl\leq \gl^{-1}M$ on $\tt$, and 
consequently, the family $\{\gl v^\gl\}_{\gl>0}$ is uniformly 
bounded on $\tt$.

By \eqref{EC}, the family $\{\gl v^\gl\}_{\gl>0}$ 
is equi-continuous on $\tt$. Setting  
\[
w^\gl:=v^\gl-\min_{\tt} v^\gl  \ \ \ \text{ on }\ \tt \ \text{ for }\ \gl>0,
\]
we observe that $\{w^\gl\}_{\gl>0}$ is relatively compact in $C(\tt)$, 
and choose a sequence $\{\gl_j\}_{j\in\N}\subset (0,\,\infty)$, converging 
to zero, so that $\{w^{\gl_j}\}_{j\in\N}$ converges in $C(\tt)$ to some 
function $u\in C(\tt)$ as $j\to\infty$. The uniform 
boundedness of $\{\gl v^\gl\}_{\gl>0}$ allows us to assume, 
after taking a subsequence if necessary, that the limit
\[
d:=\lim_{j\to\infty}\gl_j\min_{\tt}v^{\gl_j}\, \in\,\R
\]
exists. Then the equi-continuity of $\{v^\gl\}_{\gl>0}$ ensures 
that as $j\to\infty$,
\[
\gl_j v^{\gl_j}(x) \to d \ \ \ \text{ uniformly on }\tt.
\]

Now, according to the boundary condition $v^\gl\leq g$, pointwise on $\bry$,  
for any $\gl>0$, 
we have 
\begin{equation}\label{prop-d1-1}
\min_{\tt}v^\gl\leq \min_{\bry}g \ \  \ \text{ for all }\ \gl>0,
\end{equation}
and hence, we find that $d\leq 0$. 

Set 
\[
m_j:=\min_{\tt}v^{\gl_j} \ \ \ \text{ for }j\in\N,
\]
and consider the sequence $\{m_j\}_{j\in\N}$, which is bounded from above 
due to \eqref{prop-d1-1}. By passing 
to a subsequence if needed, we may assume 
that
\[
m:=\lim_{j\to\infty}m_j\in[-\infty,\,\min_{\bry}g].
\]  
Observe that, for any $j\in\N$, $w^{\gl_j}$ is a solution of 
\[
\begin{cases}
\gl_j(w^{\gl_j}+m_j)+F[w_j]=0 \ \ \text{ in }\gO,\\[3pt]
w^{\gl_j}+m_j=g \ \ \ \text{ on }\ \pl\gO. 
\end{cases}
\]
Thus, in the limit $j\to\infty$, 
we find that if $m>-\infty$,
then $w$ is a solution of (\ref{ED}$_0$), with $g$ replaced by $g-m$,
which says that the function $u:=w+m$ is a solution of 
(ED$_0$). Notice that if $m>-\infty$, then $d=0$. 

On the other hand, if $m=-\infty$, then, for $j$ sufficiently large, 
we have $v^{\gl_j}<g$ on $\bry$, which implies that $v^{\gl_j}$ is a 
supersolution of $\gl_j v^{\gl_j}+F[v^{\gl_j}]=0$ on $\tt$. 
This can be stated that $w^{\gl_j}$ is a supersolution of 
$\gl_j(w^{\gl_j}+m_j)+F[w^{\gl_j}]=0$ on $\tt$. Sending $j\to\infty$, we 
deduce that $w$ is a solution of (ES$_{-d}$). 
Note that if $w\in C(\tt)$ is a solution of (ES$_0$) 
and if $C\in \R$ is large enough so that $w-C\leq g$ on $\bry$, then 
$u:=w-C$ is a solution of (ED$_0$). 

Thus, we conclude that either problem 
(ED$_0$) has a solution in $C(\tt)$, or else
problem (ES$_c$), with some $c>0$, has a solution in $C(\tt)$.    
\end{proof}

\subsection{Representation formulas in the case $\gl>0$}
We first need to do some setup  
to take the Dirichlet boundary condition into account.

For $M>0$ and $\gl\ge 0$, we  
define $\cF^{\rD}(\gl)$ (resp., $\cF^{\rD}(M,\gl)$) as the set of 
all $(\phi,\psi,u)\in \Psi^+\tim C(\tt)$ 
(resp., $(\phi,\psi,u)\in \Psi^+(M)\tim C(\tt)$)
such that  
$u$ is a subsolution of \eqref{D'}. 
Fix $z \in \tt$ and define the sets 
$\cG^{\rD}(z,\gl),\,\cG^{\rD}(M,z,\gl)\subset C(\tt)$ by 
\[
\begin{aligned}
\cG^{\rD}(z,\gl):&\,=\left\{\big(\phi-\gl u(z), 
\gl(\psi-u(z))\big)\mid (\phi,\psi,u)\in\cF^{\rD}(\gl)\right\},\\
\cG^{\rD}(M,z,\gl):&\,=\left\{\big(\phi-\gl u(z), 
\gl(\psi-u(z))\big)\mid (\phi,\psi,u)\in\cF^{\rD}(M,\gl)\right\}. 
\end{aligned}\]
We also write $\cG^\rD(0)$ and $\cG^{\rD}(M,0)$, respectively, for $\cG^{\rD}(z,0)$ and $\cG^{\rD}(M,z,0)$, which 
are independent of $z$. 
Notice that
\[\cG^{\rD}(0)=\left\{(\phi, 0)\mid 
(\phi,\psi,u)\in\cF^{\rD}(0)\right\}\  
\text{ and }\  
\cG^{\rD}(M,0)=\left\{(\phi, 0)\mid 
(\phi,\psi,u)\in\cF^{\rD}(M,0)\right\},
\]
and that
\[
\cF^\rD(z,\gl)=\bigcup_{M>0} \cF^{\rD}(M,z,\gl) 
\ \ \ \text{ and } \ \ \  
\cG^{\rD}(z,\gl)= \bigcup_{M>0} \cG^{\rD}(M,z,\gl).
\]

\begin{lem}\label{lem:dirichlet-convex}
Assume \eqref{F1}, \eqref{F2}, and \eqref{CP}. 
For any $(z,\gl)\in\tt\tim[0,\infty)$ and $M>0$, 
the sets $\cF^\rD_{z,\gl}$, $\cF^\rD(M,z,\gl)$, $\cG^\rD(z,\gl)$ and  
$\cG^{\rD}(M, z,\gl)$ are  
convex cones with vertex at the origin.
\end{lem}

\begin{proof} It is easily seen that $\Psi^+$ and $\Psi^+(M)$ are convex cones with vertex at the origin. 

For $i=1,2$, let $(\phi_i,\psi_i,u_i)\in\cF^{\rD}(\gl)$, 
fix $t,s\in(0,\,\infty)$ and set 
\[
u=tu_1+su_2, \quad \phi=t\phi_1+s\phi_2 \ \ 
\text{ and } \ \ \psi=t\psi_1+ s\psi_2.
\]
As in the proof of \cite[Lemma 2.8]{IsMtTr1}, we find that $u$ 
is a subsolution of $\gl u+F_\phi[u]=0$ in $\gO$. Since 
$u_i\leq \psi_i$ pointwise in $\bry$ for $i=1,2$, we get immediately 
$u\leq \psi$ pointwise in $\bry$. Hence, we see that
$(\phi,\psi,u)\in\cF^\rD(\gl)$. 

Assume, in addition, that 
$(\phi_i,\psi_i,u_i)\in\cF^\rD(M,\gl)$ 
for $i=1,2$. Then  $(\phi_i,\psi_i)\in\Psi^+(M)$ for $i=1,2$, 
and, hence, the cone property of $\Psi^+(M)$ implies that 
$(\phi,\psi)\in\Psi^+(M)$, which proves, together with the property 
of $(\phi,\psi,u)$ being in $\cF^\rD(\gl)$, 
that $(\phi,\psi,u)\in\cF^\rD(M,\gl)$.   

Thus, we see that 
$\cF^{\rD}(\gl)$ and $\cF^{\rD}(M,\gl)$ (and also 
$\cG^{\rD}(z,\gl)$ and $\cG^{\rD}(M,z,\gl)$) are 
convex cones with vertex at the origin. 
\end{proof}

Henceforth, we write 
\[
\cR_{1}:=\cR_{\bb},\quad 
\cR_2:=\cR_{\bry},\quad \cR_L^+:=\cR_{\bb}^+\cap\cR_L
 \ \ \text{ and } \ \ 
\quad\cR_2^+:=\cR_{\bry}^+.
\]
Define
\[
\cP^{\rD}:=\left\{(\mu_1,\mu_2)\in\cR_{L}^+\times\cR_{2}^{+}\mid 
\mu_1(\tt\times \cA)+\mu_2(\bry)=1\right\},
\]
and, for any compact subset $K$ of $\cA$,  
\[
\cP^{\rD}_K:=\left\{(\mu_1,\mu_2)\in\cP^{\rD}\mid 
\mu_1(\tt\times K)+\mu_2(\bry)=1\right\}.
\]

We define the dual cones 
$\cG^{\rD}(z,\gl)',\,\cG^\rD(M,z,\gl)'$
of $\cG^{\rD}(z,\gl),\,\cG^{\rD}(M,z,\gl)$ in 
$\cR_{L}\times\cR_{2}$, respectively, by
\[\begin{aligned}
\cG^\rD(z,\gl)':&\,=\left\{(\mu_1,\mu_2)\in \cR_{L}\times\cR_{2}
\mid \lan\mu_1,f_1\ran+\lan\mu_2,f_2\ran\ge0 
\ \ \text{ for all }  (f_1,f_2)\in\cG^{\rD}(z,\gl)\right\},
\\
\cG^\rD(M,z,\gl)':&\,=\left\{(\mu_1,\mu_2)\in 
\cR_{L}\times\cR_{2}
\mid \lan\mu_1,f_1\ran+\lan\mu_2,f_2\ran\ge0 
\ \ \text{ for all }  (f_1,f_2)\in\cG^{\rD}(M,z,\gl)\right\}. 
\end{aligned}
\]
It is obvious that 
\[
\cG^\rD(z,\gl)'=\bigcap_{M>0}\cG^\rD(M,z,\gl)'.
\]
As usual, we write $\cG^\rD(0)'$ and $\cG^\rD(M)'$ 
for $\cG^\rD(z,0)'$ and $\cG^\rD(M,z,0)'$, respectively.  

The following proposition is a key step toward 
Theorems \ref{thm1-d0} and \ref{thm3-d}, two of the main results in this section. 

\begin{thm} \label{thm1-d} 
Assume \eqref{F1}, \eqref{F2}, \eqref{CP},  
\eqref{L} and \eqref{CPD}. 
Let $(z,\gl)\in\tt\tim (0,\,\infty)$  
and $M> \|g\|_{C(\bry)}$.  
If $v^\gl\in C(\tt)$ is a 
solution of \eqref{D}, then
\begin{equation}\label{thm1-d-0}
\gl v^\gl(z)=\inf_{(\mu_1,\mu_2)\in \cP^{\rD} \cap 
\cG^\rD(M,z,\gl)'}\,\left(
\lan \mu_1,\,L\ran+\gl\lan\mu_2,\,g\ran\right).
\end{equation}
\end{thm}

We need a lemma for the proof of the theorem above.

\begin{lem}\label{lem1-d}Assume \eqref{F1}, \eqref{F2}, 
and \eqref{CPD}. 
 Let $\gl>0$. 
If $(\phi,\psi,u)\in\cF^{\rD}(\gl)$ and $\phi=t(L+\chi)$, with $t>0$ and $\chi\in C(\tt)$, then 
\[
\gl u\leq \max\left\{t\max_{\tt}(\chi-F(\cdot,0,0)),\,\gl\max_{\bry}\psi \right\} 
\ \ \text{ on }\tt. 
\]
\end{lem}

\begin{proof} Let $(\phi,\psi,u)\in\cF^{\rD}(\gl)$, and assume that 
$\phi=t(L+\chi)$ for some $t>0$ and $\chi\in C(\tt)$. 
Recall that
\[\begin{aligned}
F_\phi(x,p,X)&\,=\sup_{\ga\in\cA}\left(-\tr a X-b\cdot p -t(L+\chi)\right)
=t\sup_{\ga\in\cA}\left(-\tr a t^{-1}X-b\cdot t^{-1}p -L\right)-t\chi
\\&\,=tF(x,t^{-1}p,t^{-1}X)-t\chi(x),
\end{aligned}
\]
to find that $u$ is a subsolution of 
\[
\gl u+tF(x,t^{-1}Du,t^{-1}D^2u)=t\chi \ \ \text{ in }\gO.
\]
Hence, the function $v:=t^{-1}u$ is a subsolution of  
\[
\gl v+F[v]=\chi \ \ \text{ in }\gO  
\]
and satisfies $v\leq t^{-1}\psi$ pointwise on $\bry$. 

We set
\[
A=\max\left\{\gl^{-1}\max_{\tt}(\chi-F(\cdot,0,0)),\  
t^{-1}\max_{\bry}\psi\right\}, 
\]
and note that 
the constant function $w:=A$ 
is a supersolution of 
\[
\gl w+F[w]=\chi \ \ \text{ in }\gO \ \ \text{ and } \ \ 
w=t^{-1}\psi \ \ \text{ on }\bry. 
\]
Hence, comparison principle \eqref{CPD} guarantees that 
$v\leq A$ on $\tt$, which yields
\[
\gl u(x)\leq \gl t A = \max\left\{ t\max_{\tt}(\chi-F(\cdot,0,0)),\, 
\gl\max_{\bry}\psi\right\}\ \ \ \text{ for all }\ x\in\tt. \qedhere
\] 
\end{proof}


\begin{proof}[Proof of Theorem \ref{thm1-d}]
Let $v^\gl\in C(\tt)$ be a solution of \eqref{D}. 
Since $\|g\|_{C(\bry)}<M$, we have 
$(L,g,v^\gl)\in\cF^{\rD}(M,\gl)$. Hence, 
owing to the definition of $\cG^{\rD}(M,z,\gl)'$, 
we get 
\[\begin{aligned}
0&\,\leq \lan \mu_1,\,L-\gl v^\gl(z)\ran+\gl \lan\mu_2,\,g-v^\gl(z)\ran
\\&\,=
\lan \mu_1,L\ran+\gl\lan \mu_2,g\ran-\gl v^\gl(z) \ \ \text{ for all }
(\mu_1,\mu_2)\in\cP^{\rD} \cap \cG^{\rD}(M,z,\gl)',
\end{aligned}\]
and, therefore,
\begin{equation}\label{thm1-d-1}
\gl v^\gl(z)\leq \inf_{(\mu_1,\mu_2)\in
\cP^{\rD} \cap \cG^{\rD}(M,z,\gl)'}
\left(\lan\mu_1,\,L\ran+\gl\lan\mu_2,\,g\ran\right).
\end{equation}

Next, we show the reverse inequality:
\begin{equation}\label{thm1-d-1+}
\gl v^\gl(z)\geq \inf_{(\mu_1,\mu_2)
\in\cP^{\rD} \cap \cG^{\rD}(M,z,\gl)'}
\left(\lan\mu_1,\,L\ran+\gl\lan\mu_2,\,g\ran\right).
\end{equation} 
For this, we suppose toward a contradiction that 
\begin{equation}\label{thm1-d-2}
\gl v^\gl(z)+\ep<\inf_{(\mu_1,\mu_2)\in\cP^{\rD}\cap 
\cG^{\rD}(M,z,\gl)'}
\left(\lan \mu_1, L\ran+\gl\lan\mu_2,\,g\ran\right),
\end{equation}
for a small $\ep>0$.

Now, we fix a compact subset $K$ of $\cA$ as follows. 
If $\cA$ is compact, then we set $K:=\cA$. Otherwise, pick 
an $\ga_0\in\cA$ and choose a constant $L_0>0$ so that
\begin{equation}\label{ga_0L_0}
\max_{x\in\tt}L(x,\ga_0)\leq L_0 \ \ \text{ and } \ \ 
(\gl+1) |v^\gl(z)|\leq L_0.
\end{equation}
Then set
\begin{equation} \label{C_1C_2}
C_1:=M+\|F(\cdot,0,0)\|_{C(\tt)}, \ \ \ C_2:=C_1(1+\gl),
\end{equation}
\begin{equation}\label{gdL_1}
\gd:=\min\Big\{\fr 12,\,\fr\ep{4\gl(M+L_0)}\Big\} \ \ \text{ and } \ \ 
L_1:=\fr{M+L_0+C_2}{\gd}.
\end{equation}
Owing to \eqref{L}, we may select a compact set $K_0\subset\cA$ 
so that
\begin{equation}\label{defK_0}
L(x,\ga)\geq \max\{L_0,\,L_1\}(\,=L_1\,) \ \ \ \text{ for all }(x,\ga)\in\tt\tim(\cA\setminus K_0).
\end{equation}
Finally, pick an $\ga_1\in\cA\setminus K_0$ and define 
the compact set $K\subset\cA$ by
\begin{equation}\label{defK}
K:=K_0\cup\{\ga_0,\ga_1\}(\,=K_0\cup\{\ga_1\}\,).
\end{equation}

It follows from  
\eqref{thm1-d-2} that  
\begin{equation}\label{thm1-d-2+}
\gl v^\gl(z)+\ep<\inf_{(\mu_1,\mu_2)\in\cP^{\rD}_{K} \cap 
\cG^{\rD}(M,z,\gl)'}
\left(\lan \mu_1, L\ran+\gl\lan\mu_2,\,g\ran\right).
\end{equation}

By Lemma \ref{lem:dirichlet-convex}, $\cG^{\rD}(M,z,\gl)$ is a 
convex cone with vertex at the origin, and hence, we get 
\[
\inf_{(f_1,f_2)\in\cG^{\rD}(M,z,\gl)}
\left(\lan \mu_1,\,f_1\ran+\gl\lan\mu_2,f_2\ran\right)=
\begin{cases} 0 \ \ &\text{ if }\ (\mu_1,\mu_2)\in\cP^{\rD}_{K}\cap\cG^{\rD}(M,z,\gl)', \\
-\infty &\text{ if }\ (\mu_1,\mu_2)\in \cP^{\rD}_{K}\setminus
\cG^{\rD}(M,z,\gl)',
\end{cases}
\]
and moreover,
\begin{equation}\begin{aligned}
\label{thm1-d-2++} 
\inf_{(\mu_1,\mu_2)\in\cP^{\rD}_{K}\cap 
\cG^{\rD}(M,z,\gl)'}
&\left(\lan \mu_1, L\ran+\gl\lan \mu_2,g\ran\right)
\\&\,\kern-30pt=
\inf_{(\mu_1,\mu_2)\in\cP^{\rD}_{K}}\
\sup_{(f_1,f_2)\in\cG^{\rD}(M,z,\gl)}\left(\lan \mu_1, L\ran+\gl\lan\mu_2,g\ran
-\lan\mu_1,f_1\ran-\gl\lan\mu_2,f_2\ran
\right)
\\&\,\kern-30pt=\inf_{(\mu_1,\mu_2)\in\cP^{\rD}_{K}}\
\sup_{(f_1,f_2)\in\cG^{\rD}(M,z,\gl)}\,\left(\lan \mu_1, L-f_1\ran+\gl\lan\mu_2,g-f_2\ran\right).
\end{aligned}
\end{equation}

Since $K$ is compact, $\cP^{\rD}_{K}$ is a compact convex subset of 
$\cR_{L}\tim\cR_2$, with the topology of weak convergence of measures. 
Also, for any $(f_1,f_2)\in\cG^\rD(M,z,\gl)$, the functional 
\[
\cP_K^\rD\ni
(\mu_1,\mu_2)\mapsto \lan \mu_1, L-f_1\ran+\gl\lan\mu_2,g-f_2\ran\in\R
\]
is continuous. Thus, Sion's minimax theorem implies
\[\begin{aligned}
\min_{(\mu_1,\mu_2)\in\cP^{\rD}_{K}}
&\ \sup_{(f_1,f_2)\in\cG^{\rD}(M,z,\gl)}\,\left(\lan \mu_1, L-f_1\ran+\gl\lan\mu_2,g-f_2\ran\right)
\\&= \sup_{(f_1,f_2)\in\cG^{\rD}(M,z,\gl)}\ \min_{(\mu_1,\mu_2)\in\cP^{\rD}_{K}}
\,\left(\lan \mu_1, L-f_1\ran+\gl\lan\mu_2,g-f_2\ran\right), 
\end{aligned}\]
which, together with \eqref{thm1-d-2+} and \eqref{thm1-d-2++}, yields
\[
\gl v^\gl(z)+\ep <
\ \sup_{(f_1,f_2)\in\cG^{\rD}(M,z,\gl)}\ \min_{(\mu_1,\mu_2)
\in\cP^{\rD}_{K}}
\,\left(\lan \mu_1, L-f_1\ran+\gl\lan\mu_2,g-f_2\ran\right). 
\]
Thus, 
we may choose $(\phi,\psi,u)\in\cF^{\rD}(M,\gl)$ and  $(t,\chi)\in (0,\,\infty)\tim C(\tt)$ 
so that $\phi=t(L+\chi)$, $\|\chi\|_{C(\tt)}< M$, and 
\begin{equation}\label{thm1-d-4}
\gl v^\gl(z)+\ep <
\min_{(\mu_1,\mu_2)\in\cP^{\rD}_{K}}
\,\left(\lan \mu_1, L-\phi+\gl u(z)\ran+\gl\lan\mu_2,g-\psi+u(z)\ran\right).
\end{equation}
Note that, by the definition of $\cF^{\rD}(M,\gl)$, 
the inequality $\|\psi\|_{C(\bry)}< tM$ is valid.

Since $(0,\gd_x)\in \cP^{\rD}_{K}$ for all $x\in\bry$, we get 
from \eqref{thm1-d-4} 
\[
\gl v^\gl(z)+\ep<\gl(g-\psi+u(z))\quad\text{on} \ \bry,
\]
which reads 
\begin{equation}\label{thm1-d-4+}
\psi< g+(u-v^\gl)(z)-\gl^{-1}\ep \ \ \text{ on }\bry.
\end{equation}
Also, since $(\gd_{(x,\ga)},0)\in\cP^\rD_K$ for all $(x,\ga)\in 
\tt\tim K$, we get  from \eqref{thm1-d-4}
\begin{equation}\label{thm1-d-4++}
(t-1)L(x,\ga)+t\chi(x)<\gl (u-v^\gl)(z)-\ep
\ \ \ \text{ for all }\ (x,\ga)\in\tt \tim K. 
\end{equation}

We show that there are a constant $\gth>0$ 
and a subsolution $w=\theta u\in C(\tt)$ to 
\begin{equation}\label{thm1-d-6}
\gl w+F[w]=-\gl (v^\gl-w)(z)-2^{-1}\gth\ep\ \ \text{ in }\gO,   
\end{equation}
and 
\begin{equation}\label{thm1-d-6+}
w\leq g -(v^\gl-w)(z)-(2\gl)^{-1}\gth\ep\ \ \ \text{ on }\bry. 
\end{equation}
Once this is complete, we get a contradiction right away. 
Indeed, the function 
$\gz:=w+(v^\gl-w)(z)+(2\gl)^{-1}\gth\ep$ is a subsolution 
of \eqref{D},
and comparison principle \eqref{CPS} 
yields $\gz\leq v^\gl$, which, after evaluation at $z$, gives 
$\,(2\gl)^{-1}\gth\ep\leq 0$. 
This is a contradiction.

Assume that $\cA$ is compact. Then we have $K=\cA$, and therefore, 
we get from \eqref{thm1-d-4++}  
\[
\phi=L+(t-1)L+t\chi<-\gl(v^\gl-u)(z)-\ep \ \ \ \text{ on }\ \bb,
\]
which ensures that the pair of $w:=u$ and $\gth:=1$ satisfies \eqref{thm1-d-6}, 
while \eqref{thm1-d-6+} for this pair 
is an immediate consequence of \eqref{thm1-d-4+}.

We assume henceforth that $\cA$ is not compact. 
We split our further argument into two cases. 
Consider first the case when $t\leq 1$. 
Recall that $K=K_0\cup\{\ga_0,\ga_1\}$. By \eqref{thm1-d-4++}, we have 
\begin{equation}\label{temp1}
(t-1)L(x,\ga_0)+t\chi(x)<\gl(u-v^\gl)(z)-\ep \ \ \ \text{ for all }\ x\in\tt.
\end{equation}
By the choice of $\ga_0$ and $L_0$, we have
\[
L(x,\ga)\geq L_0\geq L(x,\ga_0) \ \ \ \text{ for all }\ (x,\ga)\in\tt\tim(\cA\setminus K).
\]
Then we combine this with \eqref{temp1}, to get 
\[
(t-1)L(x,\ga)+t\chi(x)<\gl(u-v^\gl)(z)-\ep \ \ \ \text{ for all }\ 
(x,\ga)\in\tt\tim(\cA\setminus K),
\]
which, furthermore, yields together with \eqref{thm1-d-4++} 
\[
(t-1)L(x,\ga)+t\chi(x)<\gl(u-v^\gl)(z)-\ep \ \ \ \text{ for all }\ 
(x,\ga)\in\tt\tim\cA.
\]
From this, we observe
\[
\phi=L+(t-1)L+t\chi<L-\gl (v^\gl-u)(z)-\ep \ \text{ on }\bb,
\]
which shows together with \eqref{thm1-d-4+} the validity of  
\eqref{thm1-d-6} and \eqref{thm1-d-6+}, with $w:=u$ and $\gth:=1$. 

Secondly, we consider the case when $t> 1$. Recall the choice of $L_0$ and $K_0$, to see
\[
|\gl v^\gl(z)|\leq L_0\leq L(x,\ga) \ \ \ \text{ for all }\ (x,\ga)
\in \tt\tim (\cA\setminus K_0). 
\]
Since $\ga_1\in K\setminus K_0$, 
this and \eqref{thm1-d-4++} yield
\begin{equation}\label{temp2}
(t-1)\gl v^\gl(z)+t\chi(x)\leq (t-1)L(x,\ga_1)
+t\chi(x)<\gl (u-v^\gl)(z)-\ep,
\end{equation}
which verifies $t\chi < -t\gl v^\gl(z) + \gl u(z) -\ep$ on $\tt$ and, furthermore,
\[
\phi=t(L+\chi) < tL - t \gl v^\gl(z) + \gl u(z) - \ep \quad \text{on} \ \tt \times \cA.
\]
This shows formally that  
\[
\gl u+\cL_\ga u\leq \phi<t L-t\gl (v^\gl-t^{-1}u)(z) -\ep \ \ 
\text{ in }\tt\tim\cA,\]
from which we deduce that if we set $w:=t^{-1}u$ 
and $\gth=t^{-1}$, then
\eqref{thm1-d-6} holds. 

We continue with the case when $t>1$. 
By Lemma \ref{lem1-d}, in view of \eqref{C_1C_2}, 
we get
\[
\gl u\leq \max\{t(M+\|F(\cdot,0,0)\|_{C(\tt)}),\gl tM\}
\leq C_1(1+\gl)t= C_2t,
\]
and, accordingly,
\begin{equation}\label{temp4}
\gl (u-v^\gl)(z)\leq C_2t+L_0\leq (C_2+L_0)t.
\end{equation}
Combining this with the second inequality of \eqref{temp2}
and \eqref{defK_0}, we get 
\[
(t-1) L_1 < Mt+(L_0+C_2)t=(M+L_0+C_2)t.
\]
Using this and \eqref{gdL_1}, we compute 
\[
t-1<\fr{(M+L_0+C_2)t}{L_1}=\gd t\leq \fr t2,
\]
which shows that $t<2$, and we get from the above
\begin{equation}\label{temp3}
t-1<\gd t\leq 2\gd\leq \fr{\ep}{2\gl(M+L_0)}.
\end{equation}
Hence, we obtain
\[
(t-1)(g-v^\gl(z))\leq (t-1)(M+L_0)<\fr{\ep}{2\gl},
\]
and moreover, by \eqref{thm1-d-4+},
\[\begin{aligned}
t^{-1}\psi&\,<g+(t^{-1}u-v^\gl)(z)-(t\gl)^{-1}\ep
+(t^{-1}-1)(g-v^\gl(z))
\\&\,<g+(t^{-1}u-v^\gl)(z)-(2t\gl)^{-1}\ep \ \ \ \text{ on }\pl\gO,
\end{aligned}
\]
which shows that \eqref{thm1-d-6+} holds with $w:=t^{-1}u$ and $\gth:=t^{-1}$.
Therefore, both \eqref{thm1-d-6} and \eqref{thm1-d-6+} hold with $w:=t^{-1}u$ and $\gth:=t^{-1}$. The proof is complete.
\end{proof} 

\begin{thm} \label{thm1-d0} 
Assume \eqref{F1}, \eqref{F2}, \eqref{CP}, 
\eqref{L} and 
\eqref{CPD}. 
Let $(z,\gl)\in\tt\tim (0,\,\infty)$. 
If $v^\gl\in C(\tt)$ is a 
solution of \eqref{D}, then
\begin{equation}\label{thm1-d0-0}
\gl v^\gl(z)=\min_{(\mu_1,\mu_2)\in \cP^\rD\cap
\cG^\rD(z,\gl)'}\,\left(
\lan \mu_1,\,L\ran+\gl\lan\mu_2,\,g\ran\right).
\end{equation}
\end{thm}

We introduce the set $\cG^\rD(z,\gl)^\dag$ for 
$(z,\gl)\in\tt\tim[0,\,\infty)$  (resp., $\cG^\rD(M,z,\gl)^\dag$ for 
$(M,z,\gl)\in(0,\,\infty)\tt\tim[0,\,\infty)$ )
as the set 
of triples $(\rho,\mu_1,\mu_2)\in
[0,\,\infty)\tim\cP^\rD$ 
satisfying 
\[
0\leq t\rho+\lan\mu_1,tL+\chi-\gl u(z)\ran+\gl\lan\mu_2,\psi-u(z)\ran
\]
for all $(tL+\chi,\psi,u)\in\cF^\rD(\gl)$ (resp., $(tL+\chi,\psi,u)\in\cF^\rD(M,\gl)$), with $t>0$ and $\chi\in C(\tt)$. 
Notice that it is required here for $(\rho,\mu_1,\mu_2)\in
\cG^\rD(z,\gl)^\dag$ (resp., $(\rho,\mu_1,\mu_2)\in
\cG^\rD(M,z,\gl)^\dag$) to fulfill the condition 
$(\mu_1,\mu_2)\in\cP^\rD$. We remark that 
the inclusion $(\mu_1,\mu_2)\in
\cP^\rD\cap \cG^\rD(z,\gl)'$ (resp., $(\mu_1,\mu_2)\in
\cP^\rD\cap \cG^\rD(M,z,\gl)'$) holds  
if and only if 
$(0,\mu_1,\mu_2)\in\cG^\rD(z,\gl)^\dag$
(resp., $(0,\mu_1,\mu_2)\in\cG^\rD(M,z,\gl)^\dag$), that 
$\bigcap_{M>0}\cG^\rD(M,z,\gl)^\dag=\cG^\rD(z,\gl)^\dag$, 
and that if $N>M$, then $\cG^\rD(N,z,\gl)^\dag\subset\cG^\rD(M,z,\gl)^\dag$.

The following lemmas are useful for the proof of Theorem 
\ref{thm1-d0}.

\begin{lem}\label{cpt-d-1}
Assume \eqref{F1} and \eqref{L}. Fix $R\in\R$. 
Then, \emph{(i)}\ the functional 
\[
(\mu_1,\mu_2)\mapsto \lan\mu_1,L\ran
\]
is lower semicontinuous on $\cP^\rD$, with the topology 
of weak convergence of measures, and \emph{(ii)}\  
the set 
\[
\cP^\rD_R:=\{(\mu_1,\mu_2)\in\cP^\rD 
\mid \lan \mu_1,L\ran \leq R\}
\] 
is compact in the topology of weak convergence of measures.
\end{lem}

It should be remarked that, by definition, 
a sequence $\{\mu_1^k,\mu_2^k\}_{k\in\N}\subset 
\cR_1\tim\cR_2$ converges to a $(\mu_1,\mu_2)\in\cR_1\tim\cR_2$
weakly in the sense of measures if and only if
\[
\lim_{k\to\infty}\left(\lan\mu_1^k,\phi\ran+
\lan\mu_2^k,\psi\ran\right)
=\lan\mu_1,\phi\ran+
\lan\mu_2,\psi\ran \ \ \ \text{ for all }\ (\phi,\psi)
\in C_c(\bb)\tim C(\tt).
\]  

\begin{proof} Let $\cX$ denote the disjoint union of 
$\bb$ and $\bry$, and note that $\cX$ has a natural metric, 
for instance, the metric  $d$ on $\cX$ given by   
the formula 
\[
d(\xi,\eta)=
\begin{cases}
d_1(\xi,\eta) \ \ &\text{ if }\ \xi,\eta\in \bb,\\
d_2(\xi,\eta) &\text{ if }\ \xi,\eta\in\bry,\\
1 &\text{ otherwise},
\end{cases}
\]
where $d_1$ is the given metric on $\bb$ 
and $d_2$ is the Euclidean metric on $\R^n$. 
With this metric structure, $\cX$ is $\gs$-compact and locally compact. Note that $B\subset\cX$ is a Borel set if and only if 
$B\cap \bb$ and $B\cap\bry$ are Borel sets in $\bb$ and $\bry$, respectively. 

We define $f\mid \cX\to \R$ by
\[
f(\xi)=
\begin{cases}
L(\xi) \ &\text{ if }\ \xi\in\bb,\\
0 &\text{ if }\ \xi\in\bry,
\end{cases}\]
and set  
\[
\cP_{\cX,R}:=\{\mu\in\cP_{\cX}\mid \lan\mu, f\ran\leq R\},
\]
where $\cP_\cX$ is defined as the set of all Radon probability 
measures on $\cX$, 
Note that $f=+\infty$ at infinity, 
and observe by Lemmas \ref{basic-lsc} and \ref{basic-cpt} that, in the topology of 
weak convergence of measures, 
the functional $\mu\mapsto \lan \mu,f\ran$ is lower semicontinuous on $\cP_{\cX}$, and 
$\cP_{\cX,R}$ is compact. 

For $(\mu_1,\mu_2)\in\cP^{\cD}$, if we put
\[
\tilde\mu(B):=\mu_1(B\cap\bb)+\mu_2(B\cap\bry) 
\ \ \text{ for any Borel set }B\subset\cX, 
\]
then $\tilde \mu$ defines a (unique) Radon probability 
measure on $\cX$. 
With this notation, it is easy to see that
\[
\lan\tilde \mu, f\ran=\lan\mu_1,L\ran
\ \ \ \text{  and } \ \ \ 
\cP_{\cX,R}=\{\tilde\mu\mid (\mu_1,\mu_2)\in\cP^\rD_R\}.
\]
Hence, we conclude that, in the topology of weak convergence 
of measures, the functional $(\mu_1,\mu_2)\mapsto 
\lan\mu_1,L\ran$ 
is lower semicontinuous on $\cP^\rD$ and the set   
$\cP_R^{\rD}$ is compact. 
\end{proof}

\begin{lem} \label{d-cpt} 
Assume \eqref{F1}, \eqref{F2} and \eqref{L}. 
Let $(M,z,\gl)\in(0,\,\infty)\tim\tt\tim[0,\,\infty)$ 
and let $\{\gl_j\}_{j\in\N}\subset[0,\,\infty)$ be a sequence converging to $\gl$.  Let 
$\{(\mu_1^j,\mu_2^j)\}_{j\in\N}$ 
be a sequence of measures such that $(\mu_1^j,\mu_2^j)\in\cP^\rD\cap
\cG^\rD(M,z,\gl_j)'$ for all $j\in\N$. Assume 
that the sequence $\{\du{\mu_1^j, L}\}_{j\in\N}$ is convergent  
and that $\{(\mu_1^j,\mu_2^j)\}_{j\in\N}$ converges to some 
$(\mu_1^0,\mu_2^0)\in\cP^\rD$ weakly in the sense of measures. 
\emph{(i)}\ If we set 
\[\rho=\lim_{j\to\infty}\du{\mu_1^j,L}-\du{\mu_1^0,L},\] 
then $(\rho,\mu_1^0,\mu_2^0)\in\cG^\rD(M,z,\gl)^\dag$. 
\emph{(ii)}\ Assume in addition 
that $(\rho,\mu_1^0,\mu_2^0)\in\cG^\rD(z,\gl)^\dag$. 
If either $\cA$ is compact or $\mu_1^0\not=0$, then  
there exist a pair $(\nu_1,\nu_2)\in\cP^\rD\cap
\cG^\rD(z,\gl)'$ of measures 
such that $\,\lan\nu_1,L\ran=\rho+\du{\mu_1^0,L}\,$ and, for all $\, (\psi,\eta)\in C(\tt)\tim C(\bry)$, 
\[
\lan\mu_1^{0},\psi\ran+\du{\mu_2^{0},\eta}
=\lan\nu_1,\psi\ran+\du{\nu_2,\eta}.
\]
\end{lem}

The proof of the lemma above is similar to that of Lemma \ref{sc-cpt}, but we give the proof 
for completeness.

\begin{proof} Note first that the lower semicontinuity of 
the functional $(\mu_1,\mu_2)\mapsto\lan\mu_1,L\ran$ on $\cP^\rD$,
as claimed by Lemma \ref{cpt-d-1}, implies that $\rho\geq 0$. 

To check the property that 
$(\rho,\mu_1^0,\mu_2^0)\in\cG^\rD(M,z,\gl)^\dag$,
let $f=(f_1,f_2)\in\cG^\rD(M,z,\gl)$, and select 
$(\phi,\psi,u)\in\cF^\rD(\gl)$, $t>0$ and $\chi\in C(\tt)$ 
so that $f=(\phi-\gl u(z),\gl(\psi-u(z)))$ and $\phi=tL+\chi$. 

Let $j\in\N$ and note that 
$(\phi+(\gl_j-\gl)u,\psi,u)\in \cF^\rD(\gl_j)$. 
Hence, if $j$ is large enough, 
then $(\phi+(\gl_j-\gl)u,\psi,u)\in \cF^\rD(M, \gl_j)$ and we get  
\[\begin{aligned}
0&\,\leq \lan\mu_1^j,\phi+(\gl_j-\gl)u-\gl_j u(z)\ran+
\du{\mu_2^j,\gl_j(\psi-u(z))}
\\&\,=t\lan\mu_1^j,L\ran+
\lan\mu_1^j,\chi+(\gl_j-\gl)u-\gl_j u(z)\ran
+\gl_j\du{\mu_2^j,\psi-u(z)}.
\end{aligned}\]   
Sending $j\to\infty$ yields 
\[
0\leq t(\rho+\du{\mu_1^0,L}) 
+\lan\mu_1^0,\chi-\gl u(z)\ran+\du{\mu_2^0,\gl(\psi-u(z))},
\]
which ensures that $(\rho,\mu_1^0,\mu_2^0)\in\cG^\rD(M,z,\gl)^\dag$, 
proving assertion (i). 

Next, we assume that $(\rho,\mu_1^0,\mu_2^0)\in\cP^\rD\cap\cG^\rD(z,\gl)'$ and show the existence of $(\nu_1,\nu_2)$ having the properties 
described in assertion (ii). 

Since $(\rho,\mu_1^0,\mu_2^0)\in\cG^\rD(z,\gl)^\dag$, we have 
\begin{equation}\label{d-cpt-1}
0\leq t\rho 
+\lan\mu_1^0,tL+\chi-\gl u(z)\ran+\du{\mu_2^0,\gl(\psi-u(z))}
\end{equation}
for all $(\phi,\psi,u)\in\cF^\rD(\gl)$, where $\phi=tL+\chi$, $t>0$ and $\chi\in C(\tt)$. 

If $\cA$ is compact, the weak convergence
of $\{(\mu_1^j,\mu_2^j)\}$ implies that $\rho=0$. Thus, in this case, 
the pair $(\nu_1,\nu_2):=(\mu_1^0,\mu_2^0)$ has all the required properties. 

Now, assume that $\cA$ is not compact and $\mu_1^0\not=0$.  
Lemma \ref{mod} ensures that there is $\tilde\mu_1^0\in\cR_L^+$ 
such that $\tilde\mu_1^0(\bb)=\mu_1^0(\bb)$, 
$\lan\tilde\mu_1^0,L\ran=\rho$ and 
$\lan\tilde\mu_1^0,\psi\ran=\lan\mu_1^0,\psi\ran$ for all 
$\psi\in C(\tt)$. 
We define $(\nu_1,\nu_2)\in\cR_L^+\tim\cR_2^+$ 
by $(\nu_1,\nu_2)=(\tilde\mu_1^0,\mu_2)$. 
It is obvious that $(\nu_1,\nu_2)\in \cP^\rD$, 
$\lan\nu_1,L\ran=\rho+\du{\mu_1^0,L}$ and 
$\lan\nu_1,\psi\ran+\du{\nu_1,\eta}=\lan\mu_1^0,\psi\ran+\du{\mu_2^0,\eta}$ for all 
$(\psi,\eta)\in C(\tt)\tim C(\bry)$. These properties and inequality \eqref{d-cpt-1} imply that $(\nu_1,\nu_2)\in \cG^\rD(z,\gl)'$, 
which completes the proof.  
\end{proof}

\begin{proof}[Proof of Theorem \ref{thm1-d0}] 
Since $\cG^\rD(z,\gl)'\subset \cG^\rD(M,z,\gl)'$ 
for any $M>0$, Theorem \ref{thm1-d} yields 
\begin{equation} \label{thm1-dp-1}
\gl v^\gl(z)\leq \inf_{(\mu_1,\mu_2)\in\cP^\rD\cap\cG^\rD(z,\gl)'}
\left(\du{\mu_1,L}+\gl\du{\mu_2,g}\right).
\end{equation}

To prove the reverse inequality of \eqref{thm1-dp-1}, 
in view of Theorem \ref{thm1-d}, we may select a sequence 
$\{(\mu_1^k, \mu_2^k)\}_{k\in\N}$ of pairs of measures on $\bb$ and on $\bry$ so that for all $k\in\N$, 
\begin{equation}\label{thm1-d-00}
\gl v^\gl(z)+\fr 1k>\lan\mu_1^k,L\ran+\gl\lan\mu_2^k,g\ran
\ \ \ \text{ and } \ \ \ (\mu_1^k,\mu_2^k)\in\cP^\rD\cap\cG^\rD(k,z,\gl)'. 
\end{equation}

Thanks to Lemma \ref{cpt-d-1}, there exists a subsequence
$\{(\mu_1^{k_j},\mu_2^{k_j})\}_{j\in\N}$ 
of $\{(\mu_1^k,\mu_2^k)\}$ such that 
$\{(\mu_1^{k_j},\mu_2^{k_j})\}_{j\in\N}$ converges to some 
$(\mu_1^0,\mu_2^0)\in \cP^\rD$ weakly in the sense of 
measures. Since $\{\du{\mu_1^k,L}\}_{k\in\N}$ is bounded, we may 
assume that $\{\du{\mu_1^{k_j},L}\}_{j\in\N}$ is convergent. 

We set $\rho=\lim_{j\to\infty}\du{\mu_1^{k_j},L}-\du{\mu_1^0,L}$, and 
note that 
\begin{equation}\label{thm1-dp-2}
\gl v^\gl(z)\geq \lim_{k\to\infty}\du{\mu_1^k,L}+\gl\du{\mu_2^0,g}
=\rho+\du{\mu_1^0,L}+\gl\du{\mu_2^0,g},
\end{equation}
and, by Lemma \ref{d-cpt}, that $(\rho,\mu_1^0,\mu_2^0)\in\cG^\rD(m,z,\gl)^\dag$ for all $m\in\N$, which implies that $(\rho,\mu_1^0,\mu_2^0)\in\cG^\rD(z,\gl)^\dag$. 

If either $\cA$ is compact or $\mu_1^0\not=0$, 
then Lemma \ref{d-cpt} guarantees that there exists 
$(\nu_1,\nu_2)\in\cP^\rD\cap\cG^\rD(z,\gl)'$ such that $\du{\nu_1,L}=\rho+\du{\mu_1^0,L}$ and $\du{\nu_1,\psi}+\du{\nu_2,\eta}=\du{\mu_1^0,\psi}+\du{\mu_2^0,\eta}$ for all $(\psi,\eta)\in C(\tt)\tim C(\bry)$.
These identities combined with \eqref{thm1-dp-2} yield 
\[
\gl v^\gl(z)\geq \du{\nu_1,L}+\gl\du{\nu_2,g},
\]
which shows that \eqref{thm1-d0-0} holds, with $(\nu_1,\nu_2)$ 
being a minimizer of the right hand side of \eqref{thm1-d0-0}.

It remains the case when $\cA$ is not compact and $\mu_1^0=0$. 
To treat this case, we observe first that $z\in\bry$ and 
$\mu_2^0=\gd_z$. Indeed, otherwise, there exists $\gz\in C^2(\tt)$ such that $\du{\mu_2^0,\gz-\gz(z)}\not=0$. By replacing 
$\gz$ by a constant multiple of $\gz$, 
we may assume that $\rho+\gl\du{\mu_2^0,\gz-\gz(z)}<0$. Noting that $(\gl\gz+L+F[\gz],\gz,\gz)\in\cF^\rD(\gl)$, we get
\[
0\leq \rho+\du{\mu_1^0,\gl\gz+L+F[\gz]-\gl\gz(z)}+\gl\du{\mu_2^0,\gz-\gz(z)}=\rho+\gl\du{\mu_2^0,\gz-\gz(z)},
\]  
which contradicts the choice of $\gz$.

Thus, we have $z\in\bry$ and $\mu_2^0=\gd_z$. 
Since 
\[
\du{\mu_1^0,\phi-\gl u(z)}+\gl\du{\mu_2^0,\psi-u(z)}
=\gl(\psi-u)(z)\geq 0
\]
for any $(\phi,\psi,u)\in \cF^\rD(\gl)$, we have 
$(0,\mu_1^0,\mu_2^0)\in\cG^\rD(z,\gl)^\dag$, which implies that 
$(\mu_1^0,\mu_2^0)\in\cP^\rD\cap\cG^\rD(z,\gl)'$. Thus, 
we find from \eqref{thm1-dp-1} and \eqref{thm1-dp-2} that 
\eqref{thm1-d0-0} holds with $(\mu_1^0,\mu_2^0)$ as a minimizer of the right hand side of \eqref{thm1-d0-0}. The proof is complete.  
\end{proof}
 
\subsection{Formula for the critical value} 
According to Propositions \ref{prop1-d0} and \ref{prop1-d},
under the hypotheses \eqref{F1}, \eqref{F2}, \eqref{CPD}, 
\eqref{SLD} and \eqref{EC}, there exists a unique number 
$c^*\in[0,\,\infty)$ such that, if $c^*=0$, then 
(ED$_0$) has a solution in $C(\tt)$ 
and, if $c^*>0$, then (ES$_{c^*}$) 
has a solution in $C(\tt)$. 
We call this number $c^*$ the critical value for the 
vanishing discount problem for \eqref{D} and denote it by $c_\rD$.    
Then, owing to Proposition \ref{prop1-d0}, we have 
\begin{equation}\label{lim-cD}
\lim_{\gl\to 0+}\gl v^\gl(x)=-c_\rD \ \ \ \text{ uniformly on }\ \tt, 
\end{equation}
where $v^\gl$ is the solution of \eqref{D}.

The following proposition gives a formula for the critical value $c_\rD$ similar to \eqref{critical-value-s}. 

\begin{prop}\label{prop-d-cv} Assume \eqref{F1},
\eqref{F2}, \eqref{L}, \eqref{CPD}, 
\eqref{SLD}, and \eqref{EC}. Then,
\begin{equation}\label{critical-value-d-1}
c_\rD=\min\{c\geq 0\mid \emph{(ED$_c$)}
\text{ has a solution in }C(\tt)\}.
\end{equation}
\end{prop}

\begin{proof} We set
\[
d:=\inf\{c\geq 0\mid \text{(ED$_c$)}
\text{ has a solution in }C(\tt)\}.
\]

By the definition of $c_\rD$ above, which is based on  Propositions \ref{prop1-d} and \ref{prop1-d0}, we have 
\begin{equation}\label{cD=lim}
\lim_{\gl\to 0}(-\gl v^\gl(x))=c_\rD \ \ \ \text{ uniformly on }\
\tt.
\end{equation}
Moreover, Proposition \ref{prop1-d} 
ensures that, if $c_\rD=0$, then (ED$_{c_{\rD}}$) has a solution in $C(\tt)$ and, if $c_\rD>0$, then 
(ES$_{c_\rD}$) has 
a solution in $C(\tt)$.  
Let $u\in C(\tt)$ be a solution of (ED$_0$) and 
(ES$_{c_\rD}$), respectively, if $c_\rD=0$ and 
$c_\rD>0$. Note that, if $c_\rD>0$, then the function 
$u-C$, with $C>0$ chosen so large that $u-C\leq g$ on $\bry$, 
is a solution of (ED$_{c_\rD}$). 
It is now clear that $d\leq c_\rD$. 

It is enough to show that $d\geq c_\rD$. 
Suppose to the contrary that $d<c_\rD$, and get a contradiction. 
We choose $(v,c)\in C(\tt)\tim[d,c_\rD)$ so that 
$v$ is a subsolution of (ED$_c$). 
Fix any $C>0$ so that 
$u-C\leq g$ on $\bry$. Since $c_\rD>0$, $u$ is a solution of (ES$_{c_\rD}$), and the function $w:=u-C$ is a supersolution of 
(ED$_{c_\rD}$). 
Fix $c_0\in(c,\,c_\rD)$ and select 
$\gl>0$ sufficiently small so that $\gl v\leq c_0-c$ 
and $-\gl w\leq c_\rD-c_0$ on $\tt$, which means that 
$v$ and $w$ are a subsolution and a supersolution, respectively, of 
\eqref{D}, with $L$ replaced by $L+c_0$. 
By \eqref{CPD}, we get $v\leq w=u-C$ on $\tt$, but this gives a 
contradiction when $C$ is sufficiently large.  
\end{proof}

Another formula for the critical value $c_\rD$ 
is stated in the next theorem.

Henceforth, we write
\[
\cP^\rD_1=\{\mu_1\mid (\mu_1,\mu_2)\in\cP^\rD
\cap\cG^\rD(0)'\} \ \ \text{ and }
\ \ \cP^\rD_{1,0}=\{\mu\in\cP^\rD_1\mid \mu(\tt\tim\cA)=1\}.
\]
A crucial remark here is that 
\begin{equation}\label{equiDS}
\cP^\rD_{1,0}=\cP^\rS\cap\cG^\rS(0)'.
\end{equation}
Indeed, the argument below guarantees the validity 
of the identity above. It is obvious that, for $\phi\in C(\bb)$, 
inclusion $(\phi,\psi,u)\in \cF^\rD(0)$ holds for some 
$(\psi,u)\in C(\bb)\tim C(\bry)$ if and only if 
$(\phi,u,u)\in\cF^\rD(0)$. Hence, for $\phi\in C(\bb)$,
we have $(\phi,0)\in\cG^\rD(0)$ if and only if 
$(\phi,u,u)\in\cF^\rD(0)$ for some $u\in C(\tt)$, 
which is if and only if $(\phi,u)\in\cF^\rS(0)$ 
for some $u\in C(\tt)$, and moreover, this is if and only if 
$\phi\in\cG^\rS(0)$.  Using these observations, 
it is easy to see that, for any $\mu\in \cP_L$, 
$\mu\in\cP^\rS\cap\cG^\rS(0)'$ if and only if 
$\mu\in\cG^\rS(0)'$, which is if and only if 
$\lan\mu,\phi\ran\geq 0$ \ for all $\, (\phi,0)\in\cG^\rD(0)$.
This is equivalent to the condition that
$(\mu,0)\in\cP^\rD\cap\cG^\rD(0)'$, which verifies 
that \eqref{equiDS} is valid.

\begin{thm} \label{thm3-d} 
Assume \eqref{F1}, \eqref{F2}, \eqref{CP},   
\eqref{L}, \eqref{CPD}, \eqref{SLD}, and \eqref{EC}.  
Then,
\begin{equation}\label{thm3-d-main}
-c_\rD
=\min_{\mu\in\cP^\rD_1}\lan\mu,L\ran.
\end{equation}
Furthermore, if $c_\rD>0$, then
\begin{equation}\label{thm3-d-main'}
-c_\rD=\min_{\mu\in\cP^\rS\cap\cG^\rS(0)'}
\lan\mu,L\ran.
\end{equation}
\end{thm}

\begin{proof} 
We show first that \eqref{thm3-d-main} is valid. 
Let $(\mu_1,\mu_2)\in\cP^\rD\cap\cG^\rD(0)'$ and $v\in C(\tt)$ 
be a solution of (ED$_{c_\rD}$). 
The existence of 
such a $v$ is guaranteed by Proposition \ref{prop-d-cv}.  
Noting that $(L+c_\rD,v,v)\in \cF^\rD(0)$, we get 
\[
0\leq \lan\mu_1,L+c_\rD\ran\leq \lan\mu_1,L\ran+c_\rD,
\]
which yields
\begin{equation}\label{thm3-d-mm}
-c_\rD\leq \inf_{\mu\in\cP^\rD_1}\lan\mu,L\ran.
\end{equation}

Pick a point $z\in\tt$ and a sequence $\{\gl_j\}_{j\in\N}
\subset(0,\,\infty)$ converging to zero. 
For $j\in\N$, owing to Theorem \ref{thm1-d0}, 
we select $(\mu_1^j,\mu_2^j)\in \cP^{\rD}\cap 
\cG^\rD(z,\gl_j)'$ so that 
\begin{equation}\label{thm3-d-0}
\gl_j v^{\gl_j}(z)=\lan\mu_1^j,L\ran
+\gl_j\lan\mu_2^j,g\ran.
\end{equation}
Note that the sequence $\{\lan\mu_1^j,L\ran\}_{j\in\N}$ is bounded.

We apply Lemma \ref{cpt-d-1}, with $\gl=0$, to find that 
there is a subsequence of $\{(\gl_{j},\mu_1^j,\mu_2^j)\}$, 
which we denote again by the same symbol, such that the sequence $\{(\mu_1^j, \mu_2^j)\}_{j\in\N}$ converges weakly in the sense of measures  
to some $(\mu_1^0,\mu_2^0)\in\cP^\rD$, 
the sequence $\{\lan\mu_1^j,L\ran\}_{j\in\N}$ is convergent, 
and
\[
\lim_{j\to\infty}\lan\mu_1^j,L\ran\geq \lan\mu_1^0,L\ran.
\]
Set
\[
\rho^0=\lim_{j\to\infty}
\du{\mu_1^j,L}-\lan\mu_1^0,L\ran,
\]
and note by Lemma \ref{d-cpt} that 
$(\rho^0,\mu_1^0,\mu_2^0)\in \cG^\rD(0)^\dag$. Note also that 
if $\cA$ is compact, then $\rho^0=0$. 
From \eqref{thm3-d-0} 
and \eqref{cD=lim}, we get
\begin{equation}\label{thm3-d-00}
-c_\rD=\rho^0+\lan\mu_1^0,L\ran. 
\end{equation}

If $\rho^0=0$, then we have 
\begin{equation}\label{thm3-d-m4}
-c_\rD=\lan\mu_1^0,L\ran,
\end{equation} 
and $(\mu_1^0,\mu_2^0)\in\cP^\rD\cap\cG^\rD(0)'$. 
These observations combined with 
\eqref{thm3-d-mm} show that, if $\rho^0=0$, \eqref{thm3-d-main} holds.

Assume instead that $\rho^0\not=0$. Then, $\rho^0>0$ and $\cA$ is not compact. Since $c_\rD\geq 0$ by Proposition \ref{prop1-d0}, 
we see from \eqref{thm3-d-00} that $\mu_1^0\not=0$. By Lemma \ref{d-cpt}, 
there exists $(\nu_1,\nu_2)\in\cP^\rD\cap\cG^\rD(0)'$ such that 
$\rho^0+\du{\mu_1^0,L}=\du{\nu_1,L}$. This together with \eqref{thm3-d-00} 
and \eqref{thm3-d-mm}
proves that, if $\rho^0\not=0$, then \eqref{thm3-d-main} holds. 
Thus, we conclude that \eqref{thm3-d-main} is always valid.

Finally, we consider the case $c_\rD>0$ and prove 
\eqref{thm3-d-main'}. We fix any minimizer 
$\mu\in\cP^\rD_1$ of the minimization problem \eqref{thm3-d-main}, and show that 
$\mu(\bb)=1$, which means $\mu\in\cP^\rD_{1,0}$ and completes 
the proof of \eqref{thm3-d-main'}. 

The condition for $\nu\in\cR_L^+$ to be in $\cP^\rD_1$ is described 
by the inequalities $\nu(\bb)\leq 1$ and 
\begin{equation}\label{thm3-d-m5}
0\leq t\lan\nu,L\ran+\lan\nu,\chi\ran 
\end{equation}
for all $(tL+\chi,\psi,u)\in\cF^\rD(0)$, where $t>0$ and $\chi\in C(\tt)$. 
Recall that $\lan\mu,L\ran=-c_\rD<0$. Suppose for the moment 
that $\mu(\bb)<1$. By choosing $\gth>1$ so that 
$\gth\mu(\bb)\leq 1$, we get a new measure 
$\nu:=\gth\mu\in\cR_L^+$ that satisfies \eqref{thm3-d-m5}, 
which shows that $\nu\in\cP^\rD_1$,  
and $\lan\nu,L\ran=-\gth c_\rD<-c_\rD$. 
This is a contradiction, 
which proves that $\mu(\bb)=1$.   
\end{proof}

\begin{definition}We call $\hat\mu\in\cP_1^\rD$ a \emph{viscosity Mather measure} if it satisfies $\,\du{\hat\mu,L}=\inf_{\mu\in\cP_1^\rD}\du{\mu,L}$, and denote by $\cM^\rD$ the 
set of all viscosity Mather measures $\mu \in \cP_1^\rD$.  
For  $(z,\gl)\in\tt\tim(0,\,\infty)$, we denote by $\cM^{\rD}(z,\gl)$ the set of all 
measures $(\hat\mu_1,\hat\mu_2)\in\cP^\rD\cap\cG^{\rD}(z,\gl)'$ 
that satisfies  
\[
\lan \hat\mu_1,\,L\ran+\gl\lan\hat\mu_2,\,g\ran
=\inf_{(\mu_1,\mu_2)\in\cP^\rD\cap\cG^{\rD}(z,\gl)'}
\,\left(\lan \mu_1,\,L\ran+\gl\lan\mu_2,\,g\ran\right), 
\]
and call $\gl^{-1}(\mu_1,\mu_2)$, with $(\mu_1,\mu_2)\in\cM^{\rD}(z,\gl)$,  
a \emph{viscosity Green measure} associated with \eqref{D}. 
\end{definition}

\subsection{Convergence with vanishing discount}

\begin{thm}\label{thm2-d}
Assume \eqref{F1}, \eqref{F2}, \eqref{L}, \eqref{CP}, \eqref{CPD}, 
\eqref{SLD}  and \eqref{EC}. 
For each $\gl>0$, let $v^\gl\in C(\tt)$ be 
the unique solution of \eqref{D}. Then, 
the family $\{v^\gl+\gl^{-1}c_{\rD}\}_{\gl>0}$ converges,  
as $\gl\to 0$, to a function $u$ in $C(\tt)$. Furthermore, 
the function $u$ is a solution of \emph{(ED$_{0}$)} 
or \emph{(ES$_{c_\rD}$)}, 
if $c_\rD=0$ or $c_\rD>0$, respectively.  
\end{thm}

\begin{proof} By Propositions \ref{prop1-d0} and \ref{prop1-d}, we find that there is a solution $v\in C(\tt)$ of (ED$_0$) 
or (ES$_{c_\rD}$), respectively, if 
$c_\rD=0$ or $c_\rD>0$. Fix such a function $v$, and argue as in the proof of Proposition \ref{prop1-d}, to show that 
$\{v^\gl+\gl^{-1}c_\rD\}_{\gl>0}$ is uniformly bounded on $\tt$. Indeed, if $c_\rD=0$, then the functions 
$v+\|v\|_{C(\tt)}$ and $v-\|v\|_{C(\tt)}$ are a supersolution and a subsolution of 
\eqref{D} for any $\gl>0$ and \eqref{CPD}, which implies that 
\[
v-\|v\|_{C(\tt)}\leq v^\gl\leq v+\|v\|_{C(\tt)} \ \ \ \text{ on }\ \tt\ \text{ for any }\ \gl>0,
\]
and, hence, the uniform boundedness of $\{v^\gl\}_{\gl>0}$ 
on $\tt$ 
in the case $c_\rD=0$. If $c_\rD>0$ and if $\gl>0$ 
is sufficiently large, then we observe that the functions 
$v+\|v\|_{C(\tt)}-\gl^{-1}c_\rD$ and $v-\|v\|_{C(\tt)}-\gl^{-1}c_\rD$ are a supersolution and a subsolution of \eqref{D} and, therefore, that 
\[
v-\|v\|_{C(\tt)}-\gl^{-1}c_\rD
\leq v^\gl\leq v+\|v\|_{C(\tt)}-\gl^{-1}c_\rD,
\]  
which shows the uniform boundedness of $\{v^\gl+\gl^{-1}c_\rD\}_{\gl>0}$ on $\tt$.
Thus, together with \eqref{EC}, 
the family $\{v^\gl+\gl^{-1}c_\rD\}_{\gl>0}$ is relatively compact in $C(\tt)$. 

We denote by $\cU$ the set of accumulation points in $C(\tt)$ 
of $\{v^\gl+\gl^{-1}c_\rD\}$ as $\gl\to 0+$. 
The relative compactness in $C(\tt)$ of the family $\{v^\gl+\gl^{-1}c_\rD\}_{\gl>0}$ ensures that $\cU\not=\emptyset$. 
In order to prove the convergence, as $\gl\to 0$, of the whole family 
$\{v^\gl+\gl^{-1}c_\rD\}_{\gl>0}$, 
it is enough to show that $\cU$ has a unique element. 

Let $v,w\in\cU$, and we demonstrate that $v=w$. 
For this, we select sequences $\{\gl_j\}_{j\in\N}$ 
and $\{\gd_j\}_{j\in\N}$ of positive numbers converging 
to zero such that 
\[
\lim_{j\to\infty}v^{\gl_j}+\gl_j^{-1}c_\rD=v\ \  \text{ and } \ \ 
\lim_{j\to\infty}v^{\gd_j}+\gd_j^{-1}c_\rD=w \ \ \ \text{ in }\ C(\tt).
\] 
As in the last part of the proof of Proposition \ref{prop1-d}, 
we deduce that $v$ and $w$ are solutions of 
(ED$_0$) or  
(ES$_{c_\rD}$) if $c_\rD=0$ or $c_\rD>0$, respectively.

Next, fix a point $z\in\tt$ and, owing to Theorem \ref{thm1-d0}, select $(\mu_1^j,\mu_2^j)\in\cM^\rD(z,\gl_j)$ for every $j\in\N$
so that  
\begin{equation}\label{thm3-d-1}
\gl_j v^{\gl_j}(z)=\lan\mu_1^j,L\ran+\gl_j\lan\mu_2^j,g\ran \ \ \ \text{ for all }\ j\in\N.
\end{equation}  
Arguing as
in the proof of Theorem \ref{thm3-d}, we may assume, 
after passing to a subsequence if necessary, that there is 
$(\mu_1,\mu_2)\in\cP^\rD\cap \cG^\rD(0)'$ such that 
\begin{align}\label{thm3-d-2}
&\lim_{j\to\infty}\lan\mu_1^j,L\ran= 
\lan\mu_1,L\ran=-c_\rD,
\\&\lim_{j\to\infty}\lan\mu_1^j,\psi\ran=\lan\mu_1,\psi\ran
\ \ \ \text{ for all }\ \psi\in C(\tt),\label{thm3-d-3}
\\&\lim_{j\to\infty}\lan\mu_2^j,h\ran=\lan\mu_2,h\ran 
\ \ \ \text{ for all }\ h\in C(\bry).\label{thm3-d-4}
\end{align}

Now, $w$ is a solution of either (ED$_{c_\rD}$) 
or (ES$_{c_\rD}$), 
and thus, we have 
$(L+c_\rD+\gl_jw,w,w)\in \cF^\rD(\gl_j)$ for all $j\in\N$. Also, 
we have 
$(L-\gd_jv^{\gd_j},v^{\gd_j},v^{\gd_j})\in\cF^\rD(0)$ 
for all $j\in\N$. Hence, we get 
\begin{align}
&0\leq \lan\mu_1^j,L+c_\rD+\gl_j w-\gl_jw(z)\ran
+\gl_j\lan \mu_2^j,w-w(z)\ran,\label{thm3-d-5}
\\&0\leq \lan\mu_1,L-\gd_j v^{\gd_j}\ran.
\label{thm3-d-6}
\end{align}

Sending $j\to\infty$ in \eqref{thm3-d-5} yields 
together with \eqref{thm3-d-2}
\[
0\leq \lan\mu_1,L\ran+c_\rD\mu_1(\bb)
=c_\rD(-1+\mu_1(\bb))\leq 0,
\]
which shows that 
\begin{equation}\label{thm3-d-7}
c_\rD(-1+\mu_1(\bb))=0.
\end{equation}

Combining this, \eqref{thm3-d-2} 
and \eqref{thm3-d-6}, we compute that for any $j\in\N$,
\[
0\leq -c_\rD+\lan \mu_1^j,-\gd_jv^{\gd_j}\ran
=-\lan\mu_1,c_\rD+\gd_j v^{\gd_j}\ran,
\]
which reads 
\[
 \lan\mu_1,v^{\gd_j}+\gd_j^{-1}c_\rD\ran\leq 0.
\]
Hence, in the limit $j\to\infty$, 
\begin{equation}\label{thm3-d-8}
\lan\mu_1,w\ran\leq 0.
\end{equation}
Moreover, combining \eqref{thm3-d-1},  
\eqref{thm3-d-5} and \eqref{thm3-d-7} gives 
\[
0\leq \gl_j\left(v^{\gl_j}+
\gl_j^{-1}c_\rD
+\lan\mu_1^j,w\ran
-w(z)+\lan\mu_2^j,w-g\ran\right),
\]
from which, after dividing by $\gl_j$ and taking the limit 
$j\to\infty$, we get
\[
0\leq v(z)+\lan\mu_1,w\ran-w(z).
\]
This and \eqref{thm3-d-8} show that $w(z)\leq v(z)$. 
Because $v,w\in\cU$ and $z\in\tt$ are arbitrary, we find that 
$\cU$ is a singleton. 
\end{proof}

\section{Neumann problem} \label{sec-n}

We consider the Neumann boundary problem in this section.
As above, we relabel \eqref{DP} as \eqref{N}, and \eqref{E} as \eqref{EN}.
The letter N in \eqref{N} and \eqref{EN} indicates ``Neumann".
Let  $\gamma$ be a continuous  
vector field on $\pl\gO$ pointing outward from $\gO$, and $g\in C(\pl\gO)$ be a given 
function. The Neumann problems of interest are
\[\tag{N$_{\gl}$}\label{N}
\begin{cases}
\gl u+F[u]= 0 \ \  &\text{ in }\ \gO, \\[3pt]
\gamma\cdot Du=g \ \ &\text{ on }\ \bry,
\end{cases}
\]
for $\gl>0$, and
\begin{equation} \tag{EN}\label{EN}\begin{cases} 
F[u]=c \ \ &\text{ in } \gO,\\
\gamma\cdot Du=g \ \ &\text{ on }\pl\gO,
\end{cases}\end{equation}
As be fore, we refer as (EN$_c$) this problem with a given constant $c$.

The function $g\in C(\bry)$ is fixed 
throughout this section. We need to consider the Neumann 
boundary problem with general datum $(\gl,\phi,\psi)
\in[0,\,\infty)\tim \Phi^+\tim C(\bry)$:
\begin{equation}\tag{N$_{\gl,\phi,\psi}$}\label{N'}
\begin{cases}
\gl u+F_\phi[u]=0 \ \ \text{ in }\ \gO,&\\[3pt]
\gamma\cdot Du=\psi \ \ \  \text{ on }\ \bry, &
\end{cases}
\end{equation}

In addition to \eqref{F1}, \eqref{F2}, \eqref{L} and \eqref{EC}, 
we introduce a few  
assumptions proper to the Neumann boundary condition. 

\[\tag{OG}\label{OG}
\left\{\text{
\begin{minipage}{0.85\textwidth}
$\gO$ has $C^1$-boundary and $\gamma$ is oblique to 
$\pl\gO$ in the sense that $\gamma\cdot 
\mathbf{n}>0$ on $\pl\gO$,
where $\mathbf{n}$ denotes the outward unit normal to $\gO$.
\end{minipage}
}\right.
\]

\[\tag{CPN}\label{CPN}
\left\{\text{
\begin{minipage}{0.85\textwidth}
For any $\gl>0$ and $(\phi,\psi)\in\Psi^+$, the comparison principle holds for \eqref{N'}, 
that is, if $v\in C(\tt)$ and $w\in C(\tt)$ are a subsolution and a supersolution of \eqref{N'}, respectively, 
then the inequality $\,v\leq w\,$ holds on $\tt$.
\end{minipage}
}\right.
\]

The following is a local version of the hypothesis above.

\[\tag{CPN$_{\rm loc}$}\label{CPN'}
\left\{\text{
\begin{minipage}{0.82\textwidth}
For any $\gl>0$ and $(\phi,\psi)\in\Psi^+$,  
the localized comparison principle holds for \eqref{N'}, 
that is, for any relative open subset $V$ of $\tt$, if  
$v\in C(\lbar V)$ and $w\in C(\lbar V)$ are 
a subsolution and a supersolution, respectively,   
of 
\[
\gl u+F_\phi[u]=0 \ \ \text{ in }\ V\cap \gO
\quad\text{ and }\quad
\gamma\cdot Du=\psi \ \ \ \text{ on }\ V\cap\pl\gO,
\]
and if $\,v\leq w\,$ on $\gO\cap\pl V$, then the inequality $\,v\leq w\,$ holds on $\lbar V$.
\end{minipage}
}\right.
\]

\[\tag{SLN}\label{SLN}
\text{
\begin{minipage}{0.85\textwidth}
For any $\gl > 0$, \eqref{N} admits a solution $v^\gl \in C(\tt)$.
\end{minipage}
}.
\]

Condition \eqref{OG} guarantees that 
there exists a $C^\infty$-function $\gz$ on $\lbar\gO$ 
such that 
\begin{equation}\label{gz}
\gamma\cdot D\gz\geq 1 \ \ \ \text{  on }\ \pl\gO, 
\end{equation}
and that any classical subsolution (resp., supersolution) of 
\begin{equation}\label{eq-loc}
\gl u+F_\phi[u]=0 \ \ \text{ in }\ V\cap \gO
\quad\text{ and }\quad
\gamma\cdot Du=\psi \ \ \ \text{ on }\ V\cap\pl\gO,
\end{equation}
where $V$ is an open subset of $\R^n$, 
is also a viscosity subsolution (resp., supersolution) of 
\eqref{eq-loc}. 

Note that \eqref{CPN'} implies \eqref{CPN}. To see this, 
one may select $V$ to be $\R^n$ in \eqref{CPN'}.  
As in the boundary conditions treated above, by a rescaling 
argument, one sees that the condition  
\eqref{CPN} (reps., \eqref{CPN'}), only 
with those $\phi=L+\chi$, where 
$\chi\in C(\tt)$ is arbitrary, implies the full condition 
\eqref{CPN} (resp., \eqref{CPN'}).

\begin{prop} \label{prop1-n}
Assume \eqref{F1}, \eqref{F2}, \eqref{OG}, \eqref{CPN}, \eqref{SLN} and \eqref{EC}.
Then there exists a solution $(u,c)\in C(\tt)\times\R$ of \eqref{EN},
and such a constant $c\in \R$ is unique.
Moreover, if $v^\gl\in C(\tt)$ is a solution of \eqref{N}, then  
\begin{equation}\label{c-n}
c=-\lim_{\gl \to 0} \gl v^\gl(x) \quad \text{ uniformly on }\ \tt.  
\end{equation}
\end{prop}

We denote by $c_{\rN}$ the constant given by the proposition above and call it the critical value of \eqref{EN}. 

\begin{proof}[Outline of proof] As noted above, there exists 
a function $\gz\in C^2(\tt)$ satisfying \eqref{gz}. We may as well 
assume that $\gz\geq 0$ on $\tt$. 
Choose two positive constants $M_1$ and $M_2$
so that $|g|\leq M_1$ on $\bry$ and $|F[\pm M_1\gz]|\leq M_2$
on $\tt$, and observe that, for any $\gl>0$, 
the functions $M_1\gz+\gl^{-1}M_2$ and $-M_1\gz-\gl^{-1}M_2$ 
are a supersolution and a subsolution of \eqref{N}. By 
\eqref{CPN}, we get $\,|v^\gl|\leq M_1\gz+\gl^{-1}M_2$ on $\tt$. 
This shows that $\{\gl v^\gl\}_{\gl>0}$ is uniformly bounded on $\tt$. Using \eqref{EC}, we find that the family 
$\{v^\gl-m^\gl\}_{\gl>0}$, where $m^\gl:=\min_{\tt}v^\gl$, 
is relatively compact in $C(\tt)$, while, for each $\gl>0$, the function 
$u:=v^\gl-m^\gl$ is a solution of $\gl u+\gl m^\gl 
+F[u]=0$ in $\gO$ and $\gamma\cdot Du=g$ on $\bry$. 
Sending $\gl\to 0$, along an appropriate sequence, yields a solution $(v,c)\in C(\tt)\tim \R$ of \eqref{EN}. The uniqueness 
of the constant $c$ is a consequence of \eqref{CPN}.  
\end{proof}

\subsection{Representation formulas}
Let $(z,\gl)\in\tt\tim[0,\,\infty)$. 
We define the sets $\cF^{\rN}(\gl)\subset C(\tt\tim\cA)\times C(\tt)$, $\cG^{\rN}(z,\gl)\subset C(\bb)\tim C(\bry)$, respectively, by 
\begin{align*}
&\cF^{\rN}(\gl):=\left\{(\phi,\psi,u)\in \Psi^+\tim C(\tt)
\mid u \ \text{is a subsolution of} \ \eqref{N'}\right\}, \\
&\cG^{\rN}(z,\gl):=\left\{(\phi-\gl u(z), \psi)\mid (\phi,\psi,u)\in\cF^{\rN}(\gl)\right\}.
\end{align*}


\begin{lem} \label{thm2-sec3+0} 
Assume \eqref{F1}, \eqref{F2}, \eqref{OG} and \eqref{CPN'}. 
Let $(z,\gl)\in\tt\tim[0,\,\infty)$.  
The set $\cG^{\rN}(z,\gl)$ is a convex cone in $C(\bb)\tim C(\bry)$ with vertex at the origin. 
\end{lem}

The proof of this is in the same line as that of \cite[Lemma 2.8]{IsMtTr1} with help of the following lemma. Thus, we omit presenting it here.

\begin{lem}\label{convexN} Assume \eqref{F1}, \eqref{F2}, 
\eqref{OG}  and 
\eqref{CPN'}. Let $\gl\in[0,\,\infty)$ and let  
$(\phi_i,\psi_i, u_i)\in\cF^{\rN}(\gl)$, with $i=1,2$.  
For $t \in (0,1)$, set 
$\phi^t=t\phi_1+(1-t)\phi_2$ on $\bb$, 
$\psi^t=t\psi_1+(1-t)\psi_2$ on 
$\bry$ and $u^t=t u_1+(1-t)u_2$ on $\tt$. 
Then, $u^t\in\cF^\rN(\gl)$. 
\end{lem}

\begin{proof} Note first that, 
given $(\phi,\psi,u)\in\Psi^+\tim C(\tt)$, 
$(\phi,\psi,u)\in\cF^\rN(\gl)$ if and 
only if $(\phi-\gl u,\psi,u)\in\cF^\rN(0)$. 
It is then easily seen that our claim follows from the special case $\gl=0$. Thus we may assume henceforth that $\gl=0$. 
 
Let $\eta\in C^2(\tt)$ and $z\in\tt$ be such that $u^t-\eta$ takes a strict maximum at 
$z$. If $z\in\gO$, then the proof of \cite[Lemma 2.8]{IsMtTr1} ensures that 
$F_{\phi_t}[\eta](z)\leq 0$.
Thus, we only need to show that, if $z\in\pl\gO$, then we have either 
$F_{\phi^t}[\eta](z)\leq 0\,$ or $\,\gamma(z)\cdot D\eta(z)\leq\psi^t(z)$.  

To do this, we assume that $z\in\pl\gO$, 
suppose to the contrary that $F_{\phi^t}[\eta](z)>0$ and 
$\,\gamma(z)\cdot D\eta(z)>\psi^t(z)$, and obtain a contradiction. 

We choose $\ep>0$ and an open neighborhood $V$, in $\R^n$,  
of $z$ so that 
\begin{equation}\label{super-eta}
F_{\phi^t}[\eta](x)>\ep \ \ \text{ in } V\cap\lbar\gO \ \ 
\text{ and } \ \ 
\gamma(x)\cdot D\eta(x)\geq \psi^t(x) \ \ 
\text{ for all }x\in V\cap\pl\gO.  
\end{equation}

Set 
\[
w(x)=-t^{-1}(1-t)u_2(x)+t^{-1}\eta(x) \ \ \text{ for }x\in\tt,
\]
and prove that $w$ is a supersolution of 
\begin{equation}\label{superVN}
\begin{cases}
F_{\phi_1}[w]=\ep \ \ \ \ \text{ in }V\cap \gO, &\\[3pt]
\gamma\cdot Dw=\psi_1 \ \ \text{ on }V\cap\pl\gO. &
\end{cases}\end{equation} 
Once this is done, we apply the comparison principle 
\eqref{CPN'} to $u_1$ and $w$, 
to obtain 
\[
\sup_{V\cap\lbar\gO}(u_1-w)\leq \sup_{\gO\cap\pl V}(u_1-w).
\]
This gives a contradiction since $t(u_1-w)=u^t-\eta$ attains a strict maximum at $z\in V$ on $V\cap\lbar\gO$. 

To prove the viscosity property \eqref{superVN} of $w$, 
we fix $\xi\in C^2(V\cap\lbar\gO)$ 
and $y\in V\cap\lbar\gO$, and assume that $w-\xi$ 
takes a minimum at $y$. 
An immediate consequence of this is that
the function $t(1-t)^{-1}(w-\xi)$ has a minimum at $y$ and, thus, the function    
\[
u_2-(1-t)^{-1}\eta+t(1-t)^{-1}\xi
\]
has a maximum at $y$. By the viscosity property of $u_2$, we get either
\begin{equation}\label{eta-xi1}
F_{\phi_2}[(1-t)^{-1}\eta-t(1-t)^{-1}\xi](y)\leq 0,
\end{equation}
or
\begin{equation}\label{eta-xi2}
y\in\pl\gO \ \ \text{ and }\ \ \gamma(y)\cdot 
D((1-t)^{-1}\eta-t(1-t)^{-1}\xi)(y)\leq \psi_2(y).
\end{equation}

If \eqref{eta-xi1} holds, then, using \eqref{super-eta}, 
we get 
\[
\ep\leq F_{\phi^t}[\eta](y)
\leq tF_{\phi_1}[\xi](y)+(1-t)F_{\phi_2}[(1-t)^{-1}\eta-t(1-t)^{-1}\xi](y)
\leq tF_{\phi_1}[\xi](y). 
\] 
On the other hand, if \eqref{eta-xi2} holds, then, using \eqref{super-eta}, we get  
\[\begin{aligned}
\gamma(y)\cdot D\xi(y)
&\geq t^{-1}\gamma(y)\cdot D\eta(y)-t^{-1}(1-t)\psi_2(y)
\\&\geq t^{-1}\psi_t(y)-t^{-1}(1-t)\psi_2(y)=\psi_1(y).
\end{aligned}\]
These show, with help of \eqref{OG}, 
that $w$ is a (viscosity) supersolution of \eqref{superVN}, which completes the proof. 
\end{proof}

\begin{lem}\label{lem-n-3} Assume \eqref{F1}, \eqref{F2}, \eqref{OG} and 
\eqref{CPN}. Let $(\phi,\psi,u)\in\cF^\rN(\gl)$, with 
$\phi=tL+\chi$ for some $t>0$ and $\chi\in C(\tt)$. 
Then, there exists a constant $C>0$, depending only on $\gO$ and $F$, such that
\[
\gl u\leq \|\chi\|_{C(\tt)}+(1+\gl)C\|\psi\|_{C(\bry)} \ \ \ \text{ on }\ \tt.
\]
\end{lem}

\begin{proof} Let $\gz\in C^2(\tt)$ be a function that satisfies 
\[
\gamma\cdot D\gz\geq 1 \ \ \text{ on }\ \bry \ \ \ \text{ and } 
\ \ \gz\geq 0 \ \ \text{ on } \tt. 
\]

As before, we observe that $v:=t^{-1}u$ is a subsolution of 
\begin{equation}\label{lem-n-3-1}
\gl v+F[v]=t^{-1}\chi \ \ \text{ in }\gO \ \ \ \text{ and } \ \ \ 
\gamma\cdot Dv=t^{-1}\psi \ \ \text{ on }\ \bry.
\end{equation}
We set 
\[
C_1:=\max_{\tt}\gz \ \ \ \text{ and } \ \ \ 
C_2:=\max_{(x,t)\in\tt\tim[0,\,1]}|F[t\gz](x)|.
\]

We set $w:=A\gz+B$ for constants $A\geq 0$ and $B\geq 0$, 
to be fixed in a moment, 
and note that
\[
\gl w+F[w]\geq \gl B+F[A\gz] \ \ \text{ in }\ \gO,
\]
and 
\[
\gamma\cdot Dw=A\gamma\cdot D\gz\geq A \ \ \text{ on }\ \bry.
\]

Now, observe that if $A>1$, then the convexity of $F$ 
yields
\[
-C_2\leq F[\gz]\leq A^{-1}F[A\gz]+(1-A^{-1})F[0]
\leq A^{-1}F[A\gz]+C_2,
\]
and, hence,
\[
F[A\gz]\geq -2AC_2 \ \ \text{ on }\ \tt,
\]
which is obviously true also in the case when $0<A\leq 1$. 
Thus, putting 
\[
A:=t^{-1}\|\psi\|_{C(\bry)}\quad \text{ and } \quad
B:=\gl^{-1}(t^{-1}\|\chi\|_{C(\tt)}+2AC_2),
\]
we see that $w$ is a supersolution 
of \eqref{lem-n-3-1}. Then, \eqref{CPN} implies that 
$v\leq w$ on $\tt$, which reads
\[
\gl u\leq \gl t w\leq 
\gl t(AC_1+B)=\|\chi\|_{C(\tt)}
+(\gl C_1+2 C_2)\|\psi\|_{C(\bry)},
\]
which completes the proof. 
\end{proof}

We set 
\[
\cP_1^\rN:=\cP_{\bb} \ \ \ \text{ and } \ \ \ 
\cP^\rN:=\cP_1^\rN\tim\cR_{\bry}^+
\]
and, for compact subset $K$ of $\cA$, 
\[
\cP_{1,K}^\rN:=\left\{\mu_1\in\cP_{\bb} 
\mid\mu_1(\tt\times K)=1\right\}. 
\]
Let $\cG^\rN(z,\gl)'$ denote the dual cone of 
$\cG^{\rN}(z,\gl)$ in $\cR_L\tim\cR_{\bry}$, that is,  
\[
\cG^\rN(z,\gl)':=\left\{(\mu_1,\mu_2)\in 
\cR_{L}\tim\cR_{\bry}\mid \lan\mu_1,\phi\ran
+\lan\mu_2,\psi\ran\ge0, \ 
\ \text{for all } (\phi,\psi)\in\cG^{\rN}(z,\gl)\right\}. 
\] 

We use as well the notation: $\cF^\rN(0)=\cF^\rN(z,0)$, 
$\cG^\rN(0)=\cG^\rN(z,0)$, and 
$\cG^\rN(0)'=\cG^\rN(z,0)'$.

\begin{thm} \label{thm1-n} 
Assume \eqref{F1}, \eqref{F2}, \eqref{L}, \eqref{OG} and \eqref{CPN'}. 
\ \emph{(i)}\ Let $(z,\gl)\in\tt\tim(0,\,\infty)$. 
If $v^\gl\in C(\tt)$ is a 
solution of \eqref{N}, then
\begin{equation} \label{n-min}
\gl v^\gl(z)=\min_{(\mu_1,\mu_2)\in \cP^\rN\cap 
\cG^{\rN}(z,\gl)'}\,\left(
\lan \mu_1,\,L\ran+\lan\mu_2,\,g\ran\right).
\end{equation}
\emph{(ii)} Assume, in addition,  
\eqref{SLN} and \eqref{EC}. Then 
\begin{equation}\label{n-min0}
-c_\rN=\min_{(\mu_1,\mu_2)\in \cP^\rN\cap \cG^{\rN}(0)'}\,\left(
\lan \mu_1,\,L\ran+\lan\mu_2,\,g\ran\right).
\end{equation}
\end{thm}

\begin{definition}
We denote the set of minimizers of \eqref{n-min} 
and that of \eqref{n-min0}  
by $\cM^{\rN}(z,\gl)$
and $\cM^{\rN}(0)$, respectively.
We call any $(\mu_1,\mu_2) \in \cM^{\rN}(0)$ a viscosity Mather measure associated with \eqref{EN},
and any 
$\gl^{-1}(\mu_1,\mu_2)$, with $(\mu_1,\mu_2) \in \cM^{\rN}(z,\gl)$ and $\gl>0$, a viscosity Green measure 
associated with \eqref{N}.
\end{definition}

For $M>0$, $(z,\gl)\in\tt\tim (0,\infty)$, 
we define $\cF^\rN(M,\gl)$, $\cG^\rN(M,z,\gl)$ and 
$\cG^\rN(M,z,\gl)'$ by 
\[
\begin{aligned}
\cF^\rN(M,\gl):=&\,\cF^\rN(\gl)\cap(\Psi^+(M)\tim C(\tt)),&
\\\cG^\rN(M,z,\gl):=&\,\{(\phi-\gl u(z),\psi)\mid (\phi,\psi,u)\in\cF^\rN(M,\gl)\},&
\\\cG^\rN(M,z,\gl)':=&\,\{(\mu_1,\mu_2)\in\cR_L\tim\cR_{\bry}
\mid 
\lan\mu_1,f\ran+\lan\mu_1,\psi\ran\geq 0&
\\&& \kern-100pt \text{ for all }\,(f,\psi)\in\cG^\rN(M,z,\gl)\}.
\end{aligned}
\]
It is easily seen by Lemma \ref{thm2-sec3+0} that
$\cG^\rN(M,z,\gl)$ is a convex cone in $C(\bb)\tim C(\bry)$ 
with vertex at the origin. 

\begin{thm} \label{thm1-n1} 
Assume \eqref{F1}, \eqref{F2}, \eqref{OG}, \eqref{L}, and \eqref{CPN'}. 
Let $z\in\tt$ and $\gl,\, M\in(0,\,\infty)$. 
If $v^\gl\in C(\tt)$ is a 
solution of \eqref{N}, then
\begin{equation} \label{n-inf}
\gl v^\gl(z)\geq 
\inf_{(\mu_1,\mu_2)\in \cP^\rN\cap\cG^{\rN}(M,z,\gl)'}
\,\left(\lan \mu_1,\,L\ran+\lan\mu_2,\,g\ran\right).
\end{equation}  
\end{thm}

\begin{proof} Let $v^\gl\in C(\tt)$ be a solution of \eqref{N}. 
We fix any $\ep\in(0,\,1)$, and show that there exist 
$R>0$ and a compact subset $K$ of $\cA$ such that
\begin{equation} \label{n-inf-1}
\gl v^\gl(z)+\ep\geq \inf_{(\mu_1,\mu_2)\in \cP^\rN_{K,R}\cap
\cG^\rN(M,z,\gl)'}(\lan\mu_1,L\ran+\lan\mu_2,g\ran),
\end{equation} 
where 
\[
\cP^\rN_{K,R}
:=\left\{(\mu_1,\mu_2)\in\cP^\rN \mid \mu_1(\tt\tim K)=1,\, \mu_2(\bry)\leq R\right\}. 
\]
Since $\cP^\rN_{K,R}\subset \cP^\rN$, it follows from \eqref{n-inf-1} that 
\[
\gl v^\gl(z)+
\ep\geq \inf_{(\mu_1,\mu_2) \in \cP^\rN\cap\cG^\rN(M,z,\gl)'}(\lan\mu_1,L\ran+\lan\mu_2,g\ran),
\]
which implies that \eqref{n-inf} is valid.

We choose a constant $A>0$ so that $g+A\geq 1$ on $\bry$ and,
thanks to \eqref{OG} (see also \eqref{gz}), 
a function $\eta\in C^2(\tt)$ so that 
\begin{equation}\label{n-inf-2}
\gamma\cdot D\eta\leq -A \ \ \text{ on }\ \bry
\quad\text{ and } \quad\eta(z)=0.
\end{equation} 
Then, we choose constants $B>0$ and $R>0$ so that 
\begin{equation}\label{n-inf-3}
\gl\eta+F[\eta]\leq B \ \ \ \text{ on }\tt 
\quad\text{ and }\quad 
R\geq 1+B+\gl v^\gl(z).
\end{equation}

Let $K$ be a compact subset of $\cA$ to be specified later, and set
\[
I_{K,R}:=\inf_{(\mu_1,\mu_2)\in\cP^\rN_{K,R}\cap\cG^\rN(M,z,\gl)'}(\lan\mu_1,L\ran+\lan\mu_2,g\ran)
\]

Since $\cG^{\rN}(M,z,\gl)$ is a convex cone with vertex at the origin, 
we deduce that 
\[
\inf_{(f,\psi)\in\cG^{\rN}(M,z,\gl)}
\left(\lan \mu_1,\,f\ran+\lan\mu_2,\psi\ran\right)=
\begin{cases} 0 \ \ &\text{ if }\ (\mu_1,\mu_2)\in\cP^\rN_{K,R}
\cap\cG^{\rN}(M,z,\gl)', \\
-\infty &\text{ if }\ (\mu_1,\mu_2)\in \cP^\rN_{K,R}\setminus
\cG^{\rN}(M,z,\gl)'.
\end{cases}
\]
and, furthermore, 
\begin{equation}\begin{aligned}
\label{thm1-n-2+} 
I_{K,R}
=\inf_{(\mu_1,\mu_2)\in\cP^{\rN}_{K,R}}\
\sup_{(f,\psi)\in\cG^{\rN}(M,z,\gl)}\,\left(\lan \mu_1, L-f\ran+\lan\mu_2,g-\psi\ran\right).
\end{aligned}
\end{equation}

Note that $\cP^\rN_{K,R}$ is a compact, convex subset 
of $\cR_{L}\tim\cR_2^{+}$.  
Sion's minimax theorem implies that
\begin{equation}\label{n-inf-4}\begin{aligned}
I_{K,R}&\,=\min_{(\mu_1,\mu_2)\in\cP^\rN_{K,R}}
\ \sup_{(f,\psi)\in\cG^{\rN}(M,z,\gl)}\,\left(\lan \mu_1, L-f\ran+\lan\mu_2,g-\psi\ran\right)
\\&= \sup_{(f,\psi)\in\cG^{\rN}(M,z,\gl)}\ \min_{(\mu_1,\mu_2)\in\cP^\rN_{K,R}}
\,\left(\lan \mu_1, L-f\ran+\lan\mu_2,g-\psi\ran\right). 
\end{aligned}
\end{equation}

In order to prove \eqref{n-inf-1} for the fixed $R>0$ and 
a suitably chosen compact set $K\subset\cA$, we argue by contradiction. 
We suppose that
\[
\gl v^\gl(z)+\ep< I_{K,R},
\]
which,  together with  \eqref{n-inf-4}, yields 
\begin{equation}\label{thm1-n-4}
\gl v^\gl(z)+\ep <
\ \sup_{(f,\psi)\in\cG^{\rN}(M,z,\gl)}\ 
\min_{(\mu_1,\mu_2)\in\cP^\rN_{K,R}}
\,\left(\lan \mu_1, L-f\ran+\lan\mu_2,g-\psi\ran\right).
\end{equation}

Hence, we may  
choose $(\phi,\psi,u)\in\cF^{\rN}(M,\gl)$ and  
$(t,\chi)\in (0,\,\infty)\tim C(\tt)$ 
so that $\phi=tL+\chi$, $\|\chi\|_{C(\tt)}< tM$, 
$\|\psi\|_{C(\bry)}< tM$, and 
\[
\gl v^\gl(z)+\ep <
\inf_{(\mu_1,\mu_2)\in\cP^\rN_{K,R}}
\,\left(\lan \mu_1, L-\phi+\gl u(z)\ran+\lan\mu_2,g-\psi\ran\right).
\]
We get from the above 
\begin{equation}\label{thm1-n-4++}
\gl v^\gl(z)+\ep <
\inf_{\mu\in\cP^\rN_{1,K}}
\,\lan \mu, L-\phi+\gl u(z)\ran,
\end{equation}
and, since $(0,R\gd_{x})\in \cP^{\rN}_{K,R}$ for any $x\in\bry$, 
\[
\gl v^\gl(z)+\ep <
\inf_{\mu\in\cP^\rN_{1,K}}
\,\lan \mu, L-\phi+\gl u(z)\ran+R\min_{\bry}(g-\psi).
\]
Thus, setting 
\[
p:=\inf_{\mu\in\cP^\rN_{1,K}}\lan\mu, L-\phi+\gl (u-v^\gl)(z)\ran 
\ \ \text{ and } \ \ q:=\min_{\bry}(g-\psi), 
\]
we have 
\begin{equation} \label{thm1-n-4+}
\ep<p \ \ \ \text{ and } \ \ \ \ep<p+Rq.
\end{equation}

Our choice of $A$, $B$ and $\eta$ ensures that 
$(L+B,-A,\eta)\in\cF^\rN(\gl)$. Note as well that 
$(L,g,v^\gl),\ (\phi,\psi,u)\in\cF^{\rN}(\gl)$. 
Set 
\[
\nu:=\fr{R\ep}{R\ep+p}\in (0,\,1), 
\]  
and observe by the convexity of $\cF^{\rN}(\gl)$ that 
\[
(L+B\ep,(1-\ep)g-A\ep,(1-\ep)v^\gl+\ep\eta)\in\cF^{\rN}(\gl), 
\] 
and
\[
(1-\nu)(L+ B\ep,(1-\ep)g-A\ep,(1-\ep)v^\gl+\ep\eta)+\nu(\phi,\psi,u)\in\cF^{\rN}(\gl).
\]

We set 
\[\begin{gathered}
(\hat\phi,\hat \psi,\hat u):=
(1-\nu)(L+ B\ep,(1-\ep)g-A\ep,(1-\ep)v^\gl+\ep\eta)
+\nu(\phi,\psi,u),
\\
\hat t:=1+\nu(t-1) \ \ \ \text{ and } \ \ \ 
\hat \chi:=(1-\nu) B\ep+\nu\chi,
\end{gathered}
\]
and note that  
\[\begin{aligned}
\hat\phi&\,=L+\nu(\phi-L)+(1-\nu)B\ep=\hat t L+\hat\chi,
\\\hat\psi&\,=(1-\nu)(g-\ep(g+A))+\nu\psi,
\\\hat u-v^\gl&\,=\nu(u-v^\gl)+(1-\nu)\ep(\eta-v^\gl),
\\\hat t-1&\,=\nu(t-1).
\end{aligned}
\]

Using the facts that 
$g+A\geq 1$ on $\bry$ and that, by the definition of $q$,  
$\psi\leq g-q$ on $\bry$, and the second inequality of  
\eqref{thm1-n-4+}, we compute   
\[\begin{aligned}
\hat\psi&\,\leq (1-\nu)(g-\ep)+\nu(g-q)
\\&\,\leq (1-\nu)(g-\ep)+\nu\left(g+\fr{p-\ep}{R}\right)
=g-\fr{\nu\ep}{R} \ \ \text{ on }\bry.
\end{aligned}
\]
Also, we compute
\[\begin{aligned}
L-\hat\phi+&\gl(\hat u-v^\gl)(z)
\\ \,=\, &\nu(L-\phi+\gl (u-v^\gl)(z)) 
-(1-\nu)\ep(B+v^\gl(z)) \ \ \text{ in }\bb, 
\end{aligned}
\]
and, by using \eqref{thm1-n-4+} and the second 
inequality of \eqref{n-inf-3}, 
that for any $\mu\in\cP^\rN_{1,K}$,  
\[\begin{aligned}
\lan\mu&,L-\hat\phi+\gl(\hat u-v^\gl)(z)\ran
\\&\,\geq \nu p-(1-\nu)\ep(B+\gl v^\gl(z))
\\&\,=
(1-\nu)R\ep-(1-\nu)\ep(B+\gl v^\gl(z))
\geq (1-\nu)\ep. 
\end{aligned}
\]

Thus, setting $\hat\ep:=\ep\min\{R^{-1}\nu,\, 1-\nu\}$,
we find that 
\begin{equation}\label{thm1-n-5}
\hat \psi\leq g-\hat\ep \ \ \text{ on }\ \bry,
\end{equation}
and
\begin{equation}\label{thm1-n-5+}
\inf_{\mu\in\cP^\rN_{1,K}}\lan\mu, L-\hat\phi
+\gl(\hat u-v^\gl)(z)\ran>\hat\ep.
\end{equation}

We now use these estimates to show that there is positive constant $\gth$ such that $w:=\gth \hat u$ is a 
subsolution of 
\begin{equation}\label{thm1-n-6}
\begin{cases}
\gl w+F[w]=-\gl (v^\gl-w)(z)-\gth\hat\ep\ \ \text{ in }\gO&\\[3pt]
\gamma\cdot Dw\leq g \ \ \ \text{ on }\bry. 
\end{cases}
\end{equation}
After this is done, we easily get a contradiction as follows: 
observe that the function  
$\xi:=w+ (v^\gl-w)(z)+\gl^{-1}\gth\hat\ep$ is a subsolution 
of $\gl \xi+F[\xi]=0$ in $\gO$ and $\gamma \cdot D\xi=g$ on $\bry$,
and, by comparison principle \eqref{CPN}, that  
$\xi\leq v^\gl$ on $\tt$, which, evaluated at $z$, gives 
$\,\gl^{-1}\gth \hat \ep\leq 0$. We thus get a contradiction.

To show \eqref{thm1-n-6}, we treat first the case 
when $\cA$ is compact. Select $K=\cA$ 
and note that   
$\gd_{(x,\ga)}\in\cP^\rN_{1,K}$ for all $(x,\ga)\in\bb$. 
Thus, from \eqref{thm1-n-5+}, we get
\[
\hat\phi<L-\gl (v^\gl-\hat u)(z)-\hat\ep \ \ \text{ on }\bb,
\]
and we see that $\hat u$ is a subsolution of \eqref{thm1-n-6}, with $\gth=1$.

Consider next the case where $\cA$ is not compact. 
Choose a point $\ga_0\in\cA$ and a constant $L_0>0$ so that 
\begin{equation}\label{thm1-n-7}
\max_{x\in\tt}L(x,\ga_0)\leq L_0 
\ \ \ 
\text {and } \ \ \ \gl|v^\gl(z)|\leq L_0.
\end{equation}
Let $L_1>0$ be a constant to be fixed later, and, in view of 
\eqref{L}, we select a compact set $K_0\subset\cA$ so that 
\begin{equation}\label{thm1-n-8}
L(x,\ga)\geq \max\{L_0,L_1\} \ \ \ \text{ for all }\
(x,\ga)\in\tt\tim(\cA\setminus K_0).
\end{equation}
Pick a point $\ga_1\in\cA\setminus K_0$ and set $K=K_0\cup\{\ga_1\}$.
Since $\gd_{(x,\ga)}\in\tt\tim \cP_{1,K}$ for all $(x,\ga)\in\tt\tim K$,  
by \eqref{thm1-n-5+}, we get 
\begin{equation}\label{thm1-n-9}
\gl (v^\gl-\hat u)(z)+\hat\ep<(1-\hat t)L(x,\ga)-\hat\chi(x) \ \ \text{ for all }\ (x,\ga)\in\tt\tim K. 
\end{equation}

We divide the argument into two cases. 
Consider first the case when $\hat t\leq 1$. 
We repeat the same lines as in the proof of Theorem \ref{thm1-sc}, to deduce that
\[
(\hat t-1)L+\hat \chi<
L-\gl (v^\gl-\hat u)(z)-\hat \ep \ \text{ on }\bb,
\]
then that
\[
\hat\phi=L+(\hat t-1)L+\hat \chi<
-\gl (v^\gl-\hat u)(z)-\hat \ep \ \text{ on }\bb,
\]
and, hence, that $\hat u$ is a subsolution of \eqref{thm1-n-6}, with $\gth=1$.

Secondly, we consider the case when $\hat t\geq 1$. 
Again, repeating the same lines as in the proof of Theorem \ref{thm1-sc}, we obtain 
\[
\hat\phi<\hat tL-\hat t\gl v^\gl(z)+\gl \hat u(z)-\hat\ep 
\ \ \text{ in }\bb, 
\] 
which ensures that $w:={\hat t}^{-1}\hat u$ is a subsolution 
of 
\begin{equation}\label{thm1-n-10}\left\{\begin{aligned}
&\gl w+F[w]=\gl(u-v^\gl)(z)-{\hat t}^{-1}\hat\ep\ \  \text{ in }\ \gO,
\\&\gamma\cdot Dw={\hat t}^{-1}(g-\hat \ep) \ \ \text{ on }\ \bry.
\end{aligned}\right.\end{equation}

Our next step is to show that there exists $L_1$ for which  
$\hat t^{-1}(g-\hat \ep)\leq g$ on $\bry$. 
For this, we give an upper bound of $p$, independent of the choice 
of $L_1$ (and, hence, $K$). Because of \eqref{L}, we may take a positive constant $C_0$ so that $L\geq -C_0$ on $\bb$. 
Recall that $\nu(t-1)=\hat t-1>0$, and observe that
\[
L-\phi =(1-t)L-\chi
\leq (t-1)C_0+tM \ \ \text{ on }\ \bb,
\]
and, moreover, 
\[
\lan\mu, L-\phi\ran\leq (t-1)C_0+tM \ \ \text{ for all }\ \mu\in\cP^\rN_{1,K}.
\]
According to Lemma \ref{lem-n-3}, there exists 
a constant $C_1>0$, depending only on $\gO$, $F$ and $\gl$, 
such that 
\[
\gl u\leq C_1Mt \ \ \ \text{on }\ \tt.
\]
Thus, by the definition of $p$, we get
\begin{equation}\label{est-p}
p\leq (t-1)C_0+(C_1+1)Mt+\gl |v^\gl(z)|.
\end{equation}

Next, by \eqref{thm1-n-4++}, we get 
\[
(\phi-L)(x,\ga_1)\leq\gl(u-v^\gl)(z) \ \ \ \text{ for all }\ x\in\tt, 
\]
and moreover,
\[
(t-1)L_1\leq -\chi(x)+\gl(u-v^\gl)(z)
\leq (C_1+1)Mt+|\gl v^\gl(z)| \ \  \ \text{ for all }\ x\in\tt.
\]
Hence, if $L_1>(C_1+1)M$, then we get 
\begin{equation}\label{est-t}
t\leq \fr{L_1+\gl|v^\gl(z)|}{L_1-(C_1+1)M}.
\end{equation}

Now, set 
\[
p_0:=C_0+2(C_1+1)M+\gl |v^\gl(z)|,\quad\ep_0:=\min\Big\{\fr{\ep^2}{R\ep+p_0},\,\fr{\ep}{R+1}\Big\}, 
\]
and
\[
\tau:=1+\fr{\ep_0}{\|g\|_{C(\bry)}+1},
\]
and note that $p_0>0$, $0<\ep_0<1$ and $1<\tau<2$ and that these constants 
$p_0,\,\ep_0$ and $\tau$ are independent of the choice 
of $L_1$ and $K$.
Then fix $L_1>(C_1+1)M$ large enough so that
\[
\fr{L_1+\gl|v^\gl(z)|}{L_1-(C_1+1)M}\leq \tau. 
\]
Observe by \eqref{est-t} that $t\leq \tau$, which implies 
together with \eqref{est-p} and \eqref{thm1-n-4+} that
$\ep<p\leq p_0$. This implies moreover that
\[
\fr{R\ep}{R\ep+p_0}\leq \nu<\fr{R}{R+1},
\]
and thus,
\[
\hat\ep\geq \ep\min\Big\{\fr{\ep}{R\ep+p_0},\fr{1}{R+1}\Big\}=\ep_0. 
\]
Hence, 
\[\begin{aligned}
\hat\ep&\,\geq \ep_0=(\tau-1)\left(\|g\|_{C(\bry)}+1\right)
> (t-1)\|g\|_{C(\bry)}
\\&\,\geq (\hat t-1)\|g\|_{C(\bry)}
\geq -(\hat t-1)g \ \ \text{ on }\ \bry,
\end{aligned}\] 
and therefore,
\[
{\hat t}^{-1}(g-\hat\ep)<g \ \ \ \text{ on }\ \bry.
\]
We now conclude from this and \eqref{thm1-n-10} 
that $w:=\hat{t}^{-1}\hat u$ is a subsolution 
of \eqref{thm1-n-6}, with 
$\gth=\hat{t}^{-1}$, which completes the proof.  
\end{proof}

\begin{lem} \label{est-mu}
Assume \eqref{F1}, \eqref{F2}, \eqref{L} and \eqref{OG}. Let $(z,\gl,M)\in\tt\tim(0,\,\infty)\tim(0,\,\infty)$ 
and $(\mu_1,\mu_2)\in\cP^\rN\cap\cG^\rN(M,z,\gl)'$. 
For each $\ep>0$ there exists a constant 
$C>0$, depending only on $\ep$, $\gO$ and $F$, 
such that if $M>C$, then 
\[
\mu_2(\bry)\leq \ep \lan\mu_1,L\ran+C(1+\gl).
\] 
\end{lem}

\begin{proof} According to \eqref{OG} or \eqref{gz}, there exists a function $\eta\in C^2(\tt)$ such that
\[
\gamma\cdot D\eta\leq -1 \ \ \text{ on }\ \bry 
\quad\text{ and }\quad \eta\leq 0\ \ \text{ on }\ \tt.
\]
Fix any $\ep>0$, set 
\[
C_1:=\|F[\ep^{-1}\eta]\|_{C(\tt)}+\ep^{-1},
\]
and observe that if $M>C_1$, then 
\[
(L+C_1,-\ep^{-1},\ep^{-1}\eta)\in\cF^\rN(M,z,\gl). 
\]
Assume that $M>C_1$. 
Since $(\mu_1,\mu_2)\in\cP^\rN\cap\cG^\rN(M,z,\gl)'$, we get 
\[
0\leq \lan\mu_1,L+C_1-\gl \ep^{-1}\eta(z)\ran
+\lan\mu_2,-\ep^{-1}\ran,
\]
which yields
\[
\mu_2(\bry)\leq \ep \lan\mu_1,L\ran
+\ep C_1+\gl\|\eta\|_{C(\bry)}.
\]
If we set $C=(1+\ep)C_1+\gl\|\eta\|_{C(\bry)}$ 
and if $M>C$, then 
we have
\[
\mu_2(\bry)\leq \ep \lan\mu_1,L\ran
+C(1+\gl).
\]
The proof is now complete. 
\end{proof}

\begin{proof}[Proof of Theorem \ref{thm1-n}] We here prove only assertion 
(ii), since the proof of (i) is similar and slightly simpler. 

Fix a point $z\in\tt$. 
For each $\gl>0$, let $v^\gl\in C(\tt)$ be a solution of \eqref{N}. 
Recall that 
\begin{equation}\label{thm1-np-1}
-c_\rN=\lim_{\gl\to 0+}\gl v^\gl(z).
\end{equation}

By Lemma \ref{est-mu}, there exists a constant $C_1>0$,
depending only on $\|g\|_{C(\bry)}$, $\gO$ and $F$, 
such that for any $(M,\gl)\in(0,\,\infty)^2$ and  
$(\mu_1,\mu_2)\in\cP^\rN\cap\cG^\rN(M,z,\gl)'$, if
$M\geq C_1$, then 
\begin{equation}\label{thm1-np-2}
\mu_2(\bry)\leq \fr{1}{2\|g\|_{C(\bry)}+1}\lan\mu_1,L\ran+C_1(1+\gl)
\end{equation}
 
Fix a sequence $\{\gl_j\}_{j\in\N}\subset(0,\,1)$
converging to zero. 
Thanks to Theorem \ref{thm1-n1}, 
for each $j\in\N$, there exists 
$(\mu_1^j,\mu_2^j)\in\cP^\rN\cap \cG^\rN(j,z,\gl_j)'$ 
such that 
\begin{equation}\label{thm1-np-3}
\gl_j v^{\gl_j}(z)+\fr{1}{j}\geq \lan\mu_1^j,L\ran+\lan\mu_2^j,g\ran.
\end{equation}
From this, we get
\[
\lan\mu_1^j, L\ran
\leq \gl_j v^{\gl_j}(z)+1+\|g\|_{C(\bry)}\mu_2(\bry).
\]
Combine this with \eqref{thm1-np-2}, to obtain
\[
\lan\mu_1^j, L\ran\leq 2\left(\gl_j v^{\gl_j}(z)+1+
\|g\|_{C(\bry)}C_1(1+\gl_j)\right)\ \ \ \text{ if }\ j>C_1,
\]
and to find that the sequences 
\[
\{\lan\mu_1^j,L\ran\}_{j\in\N}
\quad
\text{ and }\quad\{\mu_2^j(\bry)\}_{j\in\N}
\]
are bounded from above and, hence, bounded. 

The boundedness of the two sequences above implies, with help of Lemma \ref{basic-cpt}, that 
the sequence $\{(\mu_1^j,\mu_2^j)\}_{j\in\N}$ of measures 
has a subsequence, convergent in the topology of 
weak convergence, which we denote by the same symbol. Let $(\mu_1,\mu_2)\in\cP^\rN$ be the limit of 
the sequence $\{(\mu_1^j,\mu_2^j)\}_{j\in\N}$.
By Lemma \ref{basic-lsc}, we have
\[
\lan\mu_1, L\ran\leq \liminf_{j\to\infty}\lan\mu_1^j,L\ran.
\] 
By passing again to a subsequence if necessary, 
we may assume that the sequence $\{\lan\mu_1^j,L\ran\}_{j\in\N}$
is convergent.

We set
\[
\rho:=\lim_{j\to\infty}\du{\mu_1^j,L}-\du{\mu,L}\,(\,\geq 0\,).
\] 
Sending $j\to\infty$ in \eqref{thm1-np-3} yields 
together with \eqref{thm1-np-1}
\begin{equation}\label{thm1-np-4}
-c_\rN\geq \rho+\du{\mu_1,L}+\du{\mu_2,g}.
\end{equation}
Note that if $\cA$ is compact, then $\rho=0$. 

Next, we show that $(\mu_1,\mu_2)\in\cG^\rN(0)'$. 
Pick $(tL+\chi,\psi,u)\in\cF^\rN(0)$,
where $t>0$ and $\chi\in C(\tt)$, 
and note that
$(tL+\chi+\gl_j u,\psi,u)\in\cF^\rN(M,\gl_j)$ 
for all $j\in\N$ and some constant $M>0$. The dual  
property  
of $\cG^\rN(M,\gl_j)'$ yields
\[
0\leq\du{\mu_1^j,tL+\chi+\gl_ju-\gl_ju(z)}+\du{\mu_2^j,\psi} 
\ \ \ \text{ if }\ j>M.
\]
Sending $j\to\infty$ gives 
\begin{equation}\label{thm1-np-5}
0\leq t(\rho+\du{\mu_1,L})+\du{\mu_1,\chi}+\du{\mu_2,\psi}.
\end{equation}

Owing to Lemma \ref{mod}, we may choose $\tilde \mu_1
\in\cP_1^\rN$ such that 
\[
\du{\tilde\mu_1,L}=\rho+\du{\mu_1,L}
\quad\text{ and }\quad 
\du{\tilde\mu_1,\gz}=\du{\mu_1,\gz} \ \ \text{ for all }\ 
\gz\in C(\tt).
\]
(If $\cA$ is compact, then we choose $\tilde\mu_1=\mu_1$.) 
This and \eqref{thm1-np-5} give
\[
0\leq \du{\tilde\mu_1,tL+\chi}+\du{\mu_2,\psi} 
\ \ \ \text{ for any }\ (tL+\chi,\psi,u)\in\cG^\rN(0),
\]
which ensures that $(\tilde\mu_1,\mu_2)\in
\cP^\rN\cap\cG^\rN(0)'$. Also, applying the above to 
a solution $(u,c_\rN)$ of \eqref{EN}, which implies 
$(L+c_\rN,g,u)\in\cF^\rN(0)=\cG^\rN(0)$, yields
\[
-c_{\rN}\leq\du{\tilde\mu_1,L}+\du{\mu_2,g}.
\] 
Inequality \eqref{thm1-np-4} now 
reads
\[
-c_\rN\geq \du{\tilde \mu_1,L}+\du{\mu_2,g}.
\] 
Thus, we have
\[
-c_\rN=\du{\tilde \mu_1,L}+\du{\mu_2,g}
=\inf_{(\nu_1,\nu_2)\in\cP^\rN\cap\cG^\rN(0)'}
(\du{\nu_1,L}+\du{\nu_2,g}),
\]
which finishes the proof.
\end{proof}

\subsection{Convergence with vanishing discount}

As an application of Theorem \ref{thm1-n}, we prove 
here one of main results in this section.

\begin{thm}\label{conv-n} Assume \eqref{F1},
\eqref{F2}, \eqref{OG}, \eqref{CPN'}, \eqref{SLN} and  \eqref{EC}. 
For $\gl>0$, let $v^\gl$ be a solution of \eqref{N}. 
Then 
\[
\lim_{\gl\to 0+}
(v^{\gl}+\gl^{-1}c_\rN)=u \ \ \ \text{ in }\ C(\tt)
\]
for some $u\in C(\tt)$, and $u$ is a solution of \emph{(EN$_{c_\rN}$)}.
\end{thm}

\begin{proof} Let $u_0\in C(\tt)$ be a solution of (EN$_{c_\rN}$). Observe that for each $\gl>0$, the functions 
$u_0+\|u_0\|_{C(\tt)}-\gl^{-1}c_\rN$ and $u_0-\|u_0\|_{C(\tt)}-\gl^{-1}c_\rN$ are a supersolution and a subsolution of 
\eqref{N}, respectively, and apply \eqref{CPN}, to get 
$\|v^\gl+\gl^{-1}c_\rN\|_{C(\tt)}\leq 2\|u_0\|_{C(\tt)}$ 
for all $\gl>0$, which, combined with \eqref{EC}, shows  
that the family 
$\{v^\gl+\gl^{-1}c_\rN\}_{\gl>0}$ is relatively compact in 
$C(\tt)$.  

Let $\cU$ denote the set of accumulation points of
$\{v^\gl+\gl^{-1}c_\rN\}_{\gl>0}$ in $C(\tt)$ as $\gl\to 0$.
The relative compactness of the family implies that $\cU\not=\emptyset$. Also, it is a standard observation that 
any $v\in\cU$ is a solution of 
(EN$_{c_\rN}$). 

Now, we prove that $\cU$ has a single element, which 
ensures that \eqref{conv-n} holds. 
Let $v,\, w\in\cU$, and select sequences $\{\gl_j\}_{j\in\N}$ 
and $\{\gd_j\}_{j\in\N}$ of positive numbers so that 
\[
\begin{cases}\disp
\lim_{j\to\infty}\gl_j=\lim_{j\to\infty}\gd_j=0,&\\
\disp
\lim_{j\to\infty}\left(v^{\gl_j}+\gl_j^{-1}c_\rN\right)=v \ \ \text{ in }\ C(\tt),&\\
\disp
\lim_{j\to\infty}\left(v^{\gd_j}+\gd_j^{-1}c_\rN\right)=w \ \ \text{ in }\ C(\tt).  
\end{cases}
\]
Fix any $z\in\tt$, and, owing to Theorem \ref{thm1-n}, choose 
a sequence $\{(\mu_1^j,\mu_2^j)\}_{j\in\N}$ of measures 
so that, for every $j\in\N$,
 $(\mu_1^j,\mu_2^j)\in\cP^\rN\cap \cG^\rN(z,\gl_j)'$
and
\begin{equation}\label{conv-np-1}
\gl_j v^{\gl_j}(z)=\du{\mu_1^j,L}+\du{\mu_2^j,g}.
\end{equation}

As in the proof of Theorem \ref{thm1-n}, after passing to a subsequence of $\{\gl_j,\mu_1^j,\mu_2^j\}_{j\in\N}$ if necessary,
we may assume that $\{(\mu_1^j,\mu_2^j)\}_{j\in\N}$ 
converges to some $(\mu_1,\mu_2)\in\cP^\rN$ and that 
for some $\tilde\mu_1\in\cP_1^\rN$, 
\begin{align}
&\lim_{j\to\infty}\du{\mu_1^j,L}=\du{\tilde\mu_1,L},\notag
\\&(\tilde\mu_1,\mu_2)\in\cG^\rN(0)',\notag
\\&\du{\tilde \mu_1,\psi}=\du{\mu_1,\psi} \ \ \ \text{ for all }\ \psi\in C(\tt),\notag
\\&-c_\rN=\du{\tilde\mu_1,L}+\du{\mu_2,g}.\label{conv-np-3}
\end{align}

Since $w$ is a solution of 
(EN$_{c_\rN}$), 
we have 
$(L+\gl_j w,g,w-\gl_j^{-1}c_\rN)\in\cF^\rN(\gl_j)$ for all 
$j\in\N$. Hence, by the dual cone property of 
$\cG^\rN(z,\gl_j)'$, we get 
\[
0\leq \du{\mu_1^j,L+\gl_j w-\gl_j(w(z)-\gl_j^{-1}c_\rN)}
+\du{\mu_2^j,g}.
\] 
Combining this with \eqref{conv-np-1} yields
\[
0\leq \gl_j \left(v^{\gl_j}(z)+\gl_j^{-1}c_\rN
-w(z)+\du{\mu_1^j,w}\right) \ \ \ \text{ for all }\ j\in\N,
\]
which implies in the limit $j\to\infty$ that
\begin{equation}\label{conv-np-2}
w(z)\leq v(z)+\du{\mu_1,w}=v(z)+\du{\tilde \mu_1,w}.
\end{equation}

Next, we note that $(L-\gd_j v^{\gd_j},g,v^{\gd_j})
\in\cF^\rN(0)$ and use the 
fact that $(\tilde\mu_1,\mu_2)\in\cG^\rN(0)'$ 
and then \eqref{conv-np-3}, 
to get
\[
0\leq \du{\tilde\mu_1,L-\gd_j v^{\gd_j}}
+\du{\mu_2,g}=-\gd_j\du{\tilde\mu_1,v^{\gd_j}+\gd_j^{-1}c_\rN},
\]
which yields, after division by $\gd_j$ and taking the limit 
$j\to\infty$,
\[
\du{\tilde\mu_1,w}\leq 0.
\]
From this and \eqref{conv-np-2}, we get $\ w(z)\leq v(z)$, 
and, since $z\in\tt$ and $v,w\in\cU$ are arbitrary, 
we conclude that $\cU$ has only one element. 
\end{proof}

\section{Examples} \label{sec-ex}
In this section, we give some examples to which the main theorems in this paper, 
Theorems \ref{thm2-sc}, \ref{thm2-d}, \ref{conv-n},  can be applied. 
For the case of Theorem \ref{conv-n}, we always assume that \eqref{OG} holds.

\subsection{First-order Hamilton-Jacobi equations}
We consider the first-order Hamilton-Jacobi equations \eqref{DP} 
and \eqref{E}, where the function $F$ is replaced by 
$H=H(x,p)$ on $\tt\tim\R^n$. The Hamiltonian $H$ is assumed to 
be continuous and, moreover, satisfies
two conditions: $H$ is coercive and convex, that is, 
\[
\lim_{|p|\to\infty}H(x,p)=+\infty \ \ \ \text{ uniformly for }\ x\in \tt,
\] 
and 
\[
\text{ for every $x\in\tt$, the function $p\mapsto H(x,p)$ is convex 
in $\R^n$.}
\]
We assume as well that $\gO$ has a $C^1$-boundary.  
Thanks to \cite{Ishi, BaMi}, for every $\gl>0$,
the solution $v^\gl$ of each of \eqref{S}, \eqref{D} and \eqref{N}
exists and the 
family $\{v^\gl\mid \gl>0\}$ is equi-Lipschitz on $\tt$, which, in particular, implies that \eqref{EC} holds.  
Also, the comparison principles \eqref{CP}, \eqref{CPS}, 
\eqref{CPD} and \eqref{CPN'} hold. 

We take the advantage of the equi-Lipschitz property of $\{v^\gl\mid\gl>0\}$ and replace $H$ by another convex and coercive Hamiltonian $\widetilde H\in C(\tt\tim\R^n)$ that satisfies: 
(i) $\widetilde H(x,p)=H(x,p)$ for all $(x,p)\in \tt\tim B_{M_0}$, 
where $M_0>0$ is a Lipschitz bound for $\{v^\gl\mid\gl>0\}$ and 
$B_{M_0}$ denotes the ball of radius $M_0$ with center at the origin, 
and (ii) $\lim_{|p|\to\infty}\widetilde H(x,p)/|p|=+\infty$ 
uniformly on $\tt$. Let $L$ be the Lagrangian, 
corresponding to $\widetilde H$, given by
\[
L(x,\ga)=\max_{p\in\R^n}(\ga\cdot p-\widetilde H(x,p)),
\]
which is a continuous function on $\tt\tim\R^n$. Moreover, 
we have
\[
\widetilde H(x,p)=\max_{\ga\in\R^n}(p\cdot \ga-L(x,\ga)),
\]
and
\[
\lim_{|\ga|\to\infty}L(x,\ga)=+\infty \ \ \ \text{ uniformly for }x\in\tt. 
\] 
Thus, in the present case, our choice of 
$\cA$, $a$ and $b$ are $\R^n$, $a(x,\ga)=0$ and  
$b(x,\ga)=-\ga$, respectively, and $L$ satisfies condition \eqref{L}.  
Furthermore, the solution $v^\gl$ of any of \eqref{S}, \eqref{D} or \eqref{N}, with 
$F=H$, is again a solution of the respective problem \eqref{S}, \eqref{D} or \eqref{N}, with $F$ replaced by $\widetilde H$. 

All the conditions and, hence, the convergence results 
of Theorems \ref{thm2-sc}, \ref{thm2-d}, and \ref{conv-n}, with $F=\widetilde H$, hold. This implies that the convergence assertions 
of Theorems \ref{thm2-sc}, \ref{thm2-d}, and \ref{conv-n}, with $F=H$, hold.

\subsection{Fully nonlinear, possibly degenerate elliptic equations with superquadratic growth in the gradient variable} \label{ss-superquad}
Firstly, set
\[
F_0(x,p,X):=\max_{\ga \in \cA_0} \left( - \tr\gs^t(x,\ga)\gs(x,\ga) X - b(x,\ga)\cdot p - L_0(x,\ga) \right),
\]
where $\cA_0$ is a compact subset of $\R^l$ for some $l \in \N$,
$\sigma \in \Lip(\tt \times \cA_0, \R^{k \times n})$ for some $k \in \N$, 
$b \in \Lip(\tt\times \cA_0,\R^n)$, and $L_0\in C(\tt \times \cA_0,\R)$.

Consider the operator $F$ of the form
\begin{align*}
&F(x,p,X):= 
\frac{c(x)}{m}|p|^m + F_0(x,p,X)\\
=&\, \sup_{(q,\ga)\in\R^n\times \cA_0}\left(- \tr\gs^t(x,\ga)\gs(x,\ga) X+q\cdot p  - b(x,\ga)\cdot p -\frac{c(x)^{\frac{-1}{m-1}}}{m'}|q|^{m'}- L_0(x,\ga)\right), 
\end{align*}
where $m>2$ and $c \in \Lip(\tt,\R)$ with $c>0$.
We set here $m':=m/(m-1)$. Assume further that $\partial \gO \in C^3$. 

It is worth pointing out that in this superquadratic case ($m>2$), the scale of the gradient term dominates the scale of the diffusion term.
Because of this feature, for any subsolution $u$ of $\gl u +F[u] =0$ in $\gO$ with $\gl \geq 0$ such that $u$ is bounded in $\tt$,
we have $u \in C^{0,\frac{m-2}{m-1}}(\tt)$
and the H\"older constant is dependent on 
$m$, $\|\sigma\|_{C(\tt\times \cA_0)}$, $\|b\|_{C(\tt\times \cA_0)}$, $\|L_0\|_{C(\tt\times \cA_0)}$, 
$\min_{\tt} c$, $\gl\|u^-\|_{L^\infty(\tt)}$ and $\partial \gO$
(see \cite{CaLePo}, \cite{Barl}, \cite{ArTr}).
This important point helps verifying the equi-continuity assumption \eqref{EC}.

It is clear that \eqref{F1}, \eqref{F2} and \eqref{L} hold. 
Thanks to Theorems 3.1, 3.2 and 4.1 in \cite{Barl}, \eqref{CP}, \eqref{CPS}, \eqref{CPD}, \eqref{SLS}, \eqref{SLD} and \eqref{EC} are valid.
Besides, if $u$ is a subsolution of \eqref{D'} with $\gl \geq 0$,
then $u \leq \psi$ pointwise on $\partial \gO$ (see \cite[Proposition 3.1]{DaL} or \cite[Proposition 3.1]{Barl} for the proof).

For the Neumann problem, we assume further that $\gamma \in C^2(\partial \gO)$. 
By using the ideas in \cite{Barl} with  careful modifications, we verify that
\eqref{CPN'} and \eqref{SLN} are valid.

Therefore, Theorems \ref{thm2-sc}, \ref{thm2-d}, \ref{conv-n} hold true.

\subsection{Fully nonlinear, possibly degenerate elliptic equations with superlinear growth in the gradient variable}
This example is only applicable to Neumann problem.
Let $F_0$ be as in Subsection \ref{ss-superquad}. Suppose that
\begin{align*}
&F(x,p,X):= 
\frac{|p|^m }{m}+ F_0(x,p,X)\\
=&\, \sup_{(q,\ga)\in\R^n\times \cA_0}\left(- \tr\gs^t(x,\ga)\gs(x,\ga) X+q\cdot p  - b(x,\ga)\cdot p -\frac{|q|^{m'}}{m'}- L_0(x,\ga)\right), 
\end{align*}
where $m>1$ and $m':=m/(m-1)$. Assume further that $\partial \gO \in C^{1,1}$ and $\gamma \in C^{1,1}(\partial \gO)$. 

This will be studied in details in the forthcoming paper \cite{IsMtTr3}.
It is verified there that all conditions in Theorem \ref{conv-n} are satisfied and thus Theorem \ref{conv-n} holds true.
As far as the authors know, this example has not been studied in the literature.

\subsection{Fully nonlinear, uniformly elliptic equations with linear growth in the gradient variable} \label{ss-fully}
We note that this example is only applicable to Dirichlet problem and Neumann problem.
We consider the operator
\[
F(x,p,X):= \max_{\ga \in \cA} \left(- \tr a(x,\ga) X -b(x,\ga)\cdot p-L(x,\ga)\right), 
\]
where $\cA$ is a compact metric space, $a\in \Lip(\tt \times \cA, \bS^n)$, $b\in \Lip(\tt \times \cA,\R^n)$, $L\in \Lip(\tt \times \cA,\R)$.
Assume that $a$ is uniformly elliptic, that is, 
there exists $\theta>0$ such that 
\[
\frac{1}{\theta} I \le a(x,\ga)\le \theta I \quad\text{for all} \ (x,\ga) \in\tt \times \cA, 
\]
where $I$ denotes the identity matrix of size $n$. 
It is obvious to see that \eqref{F1} and \eqref{F2} hold. 

For the Dirichlet problem, 
we assume that $\partial \gO \in C^{1,\beta}$  and $g \in C^{1,\beta}(\tt)$ for some $\beta \in (0,1)$.
Comparison principles \eqref{CP}, \eqref{CPD} are 
consequences of \cite[Theorem III.1 (1)]{IsLi}, \cite[Theorem 7.9]{CrIsLi}, respectively, under the Lipschitz 
continuity assumption on $a$, $b$ and $L$. 
By \cite[Theorem 3.1]{Tr}, \eqref{SLD} holds true. 
In view of the Krylov-Safonov estimate (see \cite{KrSa} and also \cite{Ca,Tr}), we have that
$\{v^\lambda\}_{\lambda>0}$ is equi-H\"older continuous on $\tt$, which implies (EC). 
Therefore, Theorem \ref{thm2-d} is valid. 

For the Neumann problem, we assume that $\partial \gO \in C^2$, $\gamma \in C^{0,1}(\partial \gO)$, and $g\in C^{0,\beta}(\partial \gO)$ for some $\beta \in (0,1)$.
Thanks to \cite[Theorem 7.12]{CrIsLi}, \eqref{CPN'} and  \eqref{SLN} are valid.
Furthermore, in view of \cite[Theorem A.1]{BaDa}, 
$\{v^\lambda\}_{\lambda>0}$ is equi-H\"older continuous on $\tt$, which implies (EC). 
These yield the validity of Theorem \ref{conv-n}.



\bye